\documentclass[11pt,twoside]{article}
\usepackage[margin=1.3in]{geometry}

\usepackage[T1]{fontenc}      

\usepackage{amsfonts,amsmath,amssymb,amsthm}
\usepackage{cases,color,curves}

\usepackage{cite}

\usepackage{bm,dsfont,extarrows,enumerate,graphicx,ifthen,latexsym,
            listings,makeidx,mathrsfs,psfrag,times,xcolor}

\usepackage[colorlinks,citecolor=green,linkcolor=red]{hyperref}
\usepackage[title]{appendix}

\def\MR#1{}   

\makeindex

\newtheorem{theorem}{Theorem}[section]
\newtheorem{lemma}[theorem]{Lemma}
\newtheorem{proposition}[theorem]{Proposition}

\newtheorem{corollary}[theorem]{Corollary}

\newtheorem{remark}{Remark}[section]
\newtheorem{definition}[theorem]{Definition}
\newtheorem{assumption}[theorem]{Assumption}

\begin{document}

\title{\textbf{Liouville theorems for conformal $Q$-curvature equations}}

\author{
  Meiqing Xu
  \thanks{M. Xu is partially supported by National Key R\&D Program of China 2025YFA1017600 and
NSFC 12526202.}
\and
  Hui Yang\thanks{H. Yang is partially supported by NSFC grant 12301140 and the Shanghai Frontier Science Center of Modern
Analysis.}
}

\date{\today}
\maketitle



\begin{abstract}
In this paper, we study the non-existence of positive solutions for the following conformal $Q$-curvature equation
\begin{equation*}
(-\Delta)^\sigma u = K(x) u^{\frac{n+2\sigma}{n-2\sigma}} \quad \text{in } \mathbb{R}^n,
\end{equation*}
where $ \sigma \in (0, n/2)$ is a real number. When $\sigma=1$, this equation reduces to the well-known scalar curvature equation arising from the prescribed scalar curvature problem. For general $\sigma \in (0, n/2)$, it appears in the study of prescribing $Q$-curvature. We establish Liouville theorems under various assumptions on the $Q$-curvature $K(x)$ by developing a unified approach applicable to all $\sigma \in (0, n/2)$. Our method successfully addresses the challenges posed by the absence of ODE tools in the fractional regime and the lack of a classification of Delaunay-type singular solutions for the general fractional Yamabe equation. Moreover, in the case where $K(x)$ is sign-changing, our result improves upon that of Chen-Li (Comm. Pure Appl. Math. 1995: 657-667) even in the classical setting $\sigma = 1$, by removing the asymptotic assumption on the solution at infinity.

\medskip


\medskip

\end{abstract}

\section{Introduction}
Let $(M, g)$ be an $n$-dimensional Riemannian manifold with $n \ge 3$, and let $K$ be a smooth function on $M$. The prescribed scalar curvature problem asks whether there exists a metric $\tilde{g}$ conformal to $g$ whose scalar curvature is $K$. If we denote $\tilde{g} = u^{\frac{4}{n-2}} g$
for some positive function $u$, then this problem reduces to finding positive solutions of the equation
\begin{equation}\label{1.1}
-\Delta_{g} u + c(n) R_g u = c(n) K u^{\frac{n+2}{n-2}}  \qquad\text{on $M$,}
\end{equation}
where $\Delta_{g}$ is the Laplace--Beltrami operator, $R_g$ is the scalar curvature of $(M, g)$ and $c(n)=\frac{n-2}{4(n-1)}$.  In the special case where $M = \mathbb{R}^n$ is equipped with the standard Euclidean metric, equation \eqref{1.1} reduces, after a suitable scaling, to
\begin{equation*}
-\Delta u = K(x) u^{\frac{n+2}{n-2}}  \quad \text{in } \mathbb{R}^n .
\end{equation*}
This equation has been extensively investigated; see, for example, \cite{bahri_scalar-curvature_1991,MR1938821,chang_scalar_1993,MR1131392,chen1997priori,li_prescribing_1995,li_prescribing_1998,MR4531773,MR3950490,MR1379191}.

In \cite{GJMS}, Graham, Jenne, Mason and Sparling constructed a sequence of higher order conformally invariant elliptic operators $\{ P_k^g\}$, which are based on the ambient metric construction of Fefferman and Graham \cite{ambient}. In particular, $P_1^g$ is the conformal Laplacian $L_g:=-\Delta_{g} + c(n) R_g$ and $P_2^g$ is the fourth order Paneitz operator. Up to positive constants, $P_1^g(1)$ is the scalar curvature and $P_2^g(1)$ is the fourth order $Q$-curvature of $g$.  Later, Graham and Zworski \cite{graham_scattering_2003} introduced a family of fractional order conformally invariant operators on the conformal infinity of asymptotically hyperbolic manifolds via scattering theory. Recently, Chang-Gonz\'alez \cite{chang_fractional_2011} and Case-Chang \cite{case_fractional_2016} provided a new interpretation of those fractional operators and established some interesting properties for associated fractional $Q$-curvatures.

Prescribing $Q$-curvature is a natural generalization of the classical prescribed scalar curvature problem. Both the prescribed higher order $Q$-curvature and fractional $Q$-curvature problems have attracted significant attention; see \cite{MR1338677,MR3484963,MR2456884,MR1991145,MR3420504,MR3518237,JLX,jin-nirenberg-2017,MR4525739,qing_positive_2006,MR2548028} and the references therein. On the standard Euclidean space $\mathbb{R}^n$, let $g = u^{\frac{4}{\,n-2\sigma\,}} \delta_{ij}$ be a conformal metric for some positive function $u$. Then the associated fractional conformal Laplacian $P_\sigma^g$ satisfies $P_\sigma^g(\cdot) = u^{-\frac{n+2\sigma}{n-2\sigma}} (-\Delta)^\sigma (u \cdot)$, and the fractional $Q$-curvature of $g$ is given by
\begin{equation*}
Q^{g}_{\sigma} := P_{\sigma}^g(1) = u^{-\frac{n+2\sigma}{\,n-2\sigma\,}} (-\Delta)^{\sigma} u.
\end{equation*}
Here $\sigma \in (0, \frac{n}{2})$ and $(-\Delta)^{\sigma}$ is the higher order fractional Laplacian defined as
\begin{equation*}
    (-\Delta)^\sigma=(-\Delta)^{\{\sigma\}}\circ(-\Delta)^{[\sigma]},
\end{equation*}
where $\{\sigma\}\in[0,1)$ and $[\sigma]\in\mathbb{N}$ denote the fractional part and integer part of $\sigma$, respectively. When $\sigma \in (0, 1)$, the fractional Laplacian $(-\Delta)^\sigma$ is a nonlocal operator defined by
\begin{equation*}
(-\Delta)^\sigma u(x)=C_{n,\sigma} \text{P.V.} \int_{\mathbb{R}^n} \frac{u(x)-u(y)}{|x-y|^{n+2\sigma}}dy,
\end{equation*}
where P.V. stands for the Cauchy principal value. Hence, the problem of prescribing $Q$-curvature in $\mathbb{R}^n$ is equivalent to finding positive solutions of the higher order fractional equation
\begin{equation}\label{sobolev eq 7}
(-\Delta)^\sigma u = K(x) u^{\frac{n+2\sigma}{n-2\sigma}}  \quad \text{in } \mathbb{R}^n,
\end{equation}
where $n\geq 2$ and $\sigma \in (0, n/2)$ is a real number. It is already known that the existence of positive solutions to \eqref{sobolev eq 7} depends on $K(x)$ in a rather complex manner.  In particular, the aforementioned literature has established the existence for \eqref{sobolev eq 7} under various assumptions on $K(x)$.

This paper is devoted to studying under what conditions on $K(x)$ equation \eqref{sobolev eq 7} admits no positive solution.  For the convenience of applications, we will understand the solutions to \eqref{sobolev eq 7} in the sense of distributions. For $s\in \mathbb{R}$, define
\begin{equation*}
\mathcal{L}_s (\mathbb{R}^n) =\left\{ u\in L_\text{loc}^1(\mathbb{R}^n)  \mid \int_{\mathbb{R}^n}\frac{|u(x)|}{(1+|x|)^{n+2s}}dx<\infty \right\}.
\end{equation*}
We say that $u$ is a distributional solution to \eqref{sobolev eq 7} if $u\in\mathcal{L}_\sigma (\mathbb{R}^n),\; K u^{\frac{n+2\sigma}{n-2\sigma}} \in L_\text{loc}^1(\mathbb{R}^n)$ and
\begin{equation*}
\int_{\mathbb{R}^n}u(-\Delta)^\sigma \varphi  =\int_{\mathbb{R}^n} K  u^{\frac{n+2\sigma}{n-2\sigma}}  \varphi,  \quad\text{for any }\varphi\in C_c^\infty(\mathbb{R}^n).
\end{equation*}
First of all, by the Kazdan-Warner type identity (or the Pohozaev type identity), one has
\begin{equation}\label{KW=con01}
\int_{\mathbb{R}^n} (x \cdot \nabla K(x)) u^{\frac{2n}{n-2\sigma}}(x) dx = 0,
\end{equation}
provided that $u(x)$ satisfies the fast decay rate
\begin{equation}\label{KW=con02}
u(x) \leq C |x|^{2\sigma-n}  \quad \text{near } \infty.
\end{equation}
Hence, when $K(x)$ is increasing (or decreasing) along each ray $\{t \theta | t \geq 0\}$ for any unit vector $\theta \in \mathbb{S}^{n-1}$, equation \eqref{sobolev eq 7} possesses no positive smooth solutions satisfying \eqref{KW=con02}. A natural question arises: if \eqref{KW=con02} is removed, does equation \eqref{sobolev eq 7} possess positive smooth solutions under the assumption of monotonicity on $K(x)$?  We first study this problem under the assumption that $K(x)$ is radially symmetric. Our first result is as follows.

\begin{theorem}\label{thm radial}
Let $0< \sigma <n/2$ and $u \in \mathcal{L}_\sigma (\mathbb{R}^n) \cap C(\mathbb{R}^n)$ be a nonnegative distributional solution to \eqref{sobolev eq 7}. Suppose that $K(x)=K(|x|) \in C^1[0, \infty)$ is nonnegative and radially nondecreasing. Then $u\equiv 0$ in $\mathbb{R}^n$ unless $K$ is a constant.
\end{theorem}

When $\sigma=1$, Theorem \ref{thm radial} was first proved by Ding-Ni \cite{ding_elliptic_1985} for radially symmetric solutions, and later by Bianchi \cite{gabriele_nonexistence_1997} for positive solutions. The proof by Ding-Ni is based on ODE methods, while Bianchi's proof uses the method of moving planes along with some ODE-type arguments. However, these ODE-type arguments cannot be directly applied to the fractional order equation \eqref{sobolev eq 7}. We will tackle this difficulty by establishing an integral representation equivalent to \eqref{sobolev eq 7} and combining it with the method of moving spheres as well as the method of moving planes. On the other hand, for radially nonincreasing functions $K(|x|)\geq 0$, Theorem \ref{thm radial} generally fails to hold. In particular, Ding-Ni \cite{ding_elliptic_1985} proved for $\sigma=1$ the existence of a radially nonincreasing $K(|x|)\geq 0$ such that \eqref{sobolev eq 7} has infinitely many positive solutions with the asymptotic behavior $C |x|^{-(n-2)/2}$ at infinity.

Now, turning to nonradial functions $K(x)$, we present the following Liouville-type theorem.

\begin{theorem}\label{thm Lin}
Let $K \in C^1(\mathbb{R}^n)$ be nonnegative and nondecreasing along each ray $\{t \theta | t \geq 0\}$ for any unit vector $\theta \in \mathbb{S}^{n-1}$ with
\[
 \lim_{|x| \to +\infty} K(x)=K(\infty) \in(0, +\infty].
\]
Let $0< \sigma< n/2$ and $u \in  \mathcal{L}_\sigma (\mathbb{R}^n) \cap C(\mathbb{R}^n)$ be a nonnegative distributional solution to \eqref{sobolev eq 7}. Suppose that either $\sigma$ is a positive integer or $\sigma \in (0, 1)$. Then $u\equiv 0$ in $\mathbb{R}^n$ unless $K$ is a constant.
\end{theorem}

When $\sigma=1$ and $K(\infty) \in (0, +\infty)$, Theorem \ref{thm Lin} was proved by Lin \cite{lin_liouvilletype_1998}. Lin's proof proceeded by first showing that $W(x) := |x|^{(n-2)/2} u(x)$ increases along each ray $\{t \xi \mid t \geq 0\}$ for any unit vector $|\xi| = 1$. He then used the classification (see Caffarelli-Gidas-Spruck \cite{CGS}) of singular solutions to the Yamabe equation
\[
-\Delta U= U^{\frac{n+2}{n-2}} \quad \text{in } \mathbb{R}^n \setminus \{0\}
\]
to determine the limit of $W(x)$ as $|x| \to \infty$, and finally combined this limit with the Pohozaev identity to complete the proof.
However, such a classification for the fractional Yamabe equation remains an open problem. Recently, Andrade-DelaTorre-do \'{O}-Ratzkin-Wei \cite{andrade2025classificationfractionalsingularyamabe} achieved a classification when $\sigma$ is very close to $1$. For higher order cases, classification results are only known for $\sigma=2$ and $3$ (see \cite{andrade2022classificationpositivesingularsolutions,MR3869387}). To circumvent this difficulty, we develop a method that does not rely on the classification of singular solutions to prove Theorem \ref{thm Lin}.

Another crucial  tool for establishing Theorem \ref{thm Lin} is the Pohozaev identity. The computations associated with this identity become significantly more challenging for fractional and higher order cases. In the fractional case $0< \sigma<1$, we derive a precise representation of the Pohozaev integral in terms of hypergeometric functions when $u$ decays like $|x|^{-(n-2\sigma)/2}$ at infinity. In the higher order case, we obtain an explicit iterative formula for the Pohozaev integral under the asymptotic condition $u(x) \sim |x|^{-(n-2\sigma)/2}$, based on the Pohozaev identity established by Guo-Peng-Yan \cite{guo2015existencelocaluniquenessbubbling}.

\begin{remark}
While an iterative Pohozaev-type formula is available for the higher order fractional Laplacian (see Jin-Xiong \cite{jin_asymptotic_2021}), the explicit computation of the Pohozaev integral for the profile $u(x)=c |x|^{-(n-2\sigma)/2}$, with non-integer $\sigma\in (1, n/2)$, presents substantial technical challenges. Once these computations are completed, our approach directly yields the corresponding Liouville-type theorem.
\end{remark}

Next, we establish a Liouville-type theorem for the case where $K$ is nonpositive. Such Liouville-type theorems play an important role in deriving a priori estimates of the prescribing sign-changing $Q$-curvature problem.
\begin{theorem}\label{thm negative}
Let $0< \sigma <n/2$ and $u \in \mathcal{L}_\sigma (\mathbb{R}^n) \cap C(\mathbb{R}^n)$ be a nonnegative distributional solution to \eqref{sobolev eq 7}. Suppose $K(x)\leq 0$ for any $x \in \mathbb{R}^n$, and there exist $R_0>0$, $C>0$, $\alpha >-2\sigma$ such that
\begin{equation}\label{Kx=26}
    |K(x)|\ge C|x|^{\alpha} \quad\text{for $|x|\ge R_0$}.
\end{equation}
 Then $u \equiv 0$ in $\mathbb{R}^n$.
\end{theorem}

For $\sigma=1$, Theorem \ref{thm negative} was proved by Ni \cite{ni1982elliptic} using the spherical averaging technique and an ODE-type method. However, this approach is clearly difficult to extend to the general fractional operator $(-\Delta)^\sigma$. Instead, we present a very straightforward proof for all $\sigma\in (0, n/2)$ via the method of integral equations.

Finally, we establish a Liouville-type theorem for equation \eqref{sobolev eq 7} when $K$ changes sign.

\begin{theorem}\label{thm:no-asymptotic}
Let $0< \sigma <n/2$ and $u \in \mathcal{L}_\sigma (\mathbb{R}^n) \cap C(\mathbb{R}^n)$ be a nonnegative distributional solution to \eqref{sobolev eq 7}.  Suppose that
$K(x)=K(|x|)\in C(\mathbb R^n)$ satisfies
\[
\begin{cases}
     K(r)>0 \mbox{ is nonincreasing for } 0\le r<1,\\
     K(r) \le 0  \mbox{ for  } r>1.
\end{cases}
\]
Suppose, in addition, that there exist constants $C>0$, $R_0>1$, and
$\alpha>-2\sigma$ such that
\begin{equation}\label{Qeun2608}
   |K(x)|\ge C|x|^\alpha
    \quad \text{for } |x| \geq R_0.
\end{equation}
 Then $u \equiv 0$ in $\mathbb{R}^n$.
\end{theorem}

Theorem \ref{thm:no-asymptotic} improves the result of Chen-Li \cite{CL3} even in the classical case $\sigma = 1$, by removing the following asymptotic assumption at infinity
\begin{equation}\label{eq 1.2}
\lim_{|x| \to \infty} |x|^{n-2\sigma}u(x) = M  \mbox{ for some } M>0.
\end{equation}
Similarly, it also improves the result of Chen-Li-Zhang \cite{CLZ}, which extends the result of Chen-Li \cite{CL3} to the fractional setting $0<\sigma<1$. When $0 < \sigma \leq 1$, since the operator $(-\Delta)^\sigma$ satisfies the maximum principle and the solution $u$ is assumed to have the asymptotic behavior \eqref{eq 1.2} at infinity, the proofs in Chen-Li \cite{CL3} and Chen-Li-Zhang \cite{CLZ} use the moving sphere method directly on the equation \eqref{sobolev eq 7}. However, the maximum principle fails for general higher order operators, and condition \eqref{eq 1.2} is not assumed in Theorem \ref{thm:no-asymptotic}.  So our approach is quite different from theirs. We first establish an integral representation for \eqref{sobolev eq 7} with sign-changing $K(x)$, from which we derive the decay estimate $u(x)=O\bigl(|x|^{-(n-2\sigma)}\bigr)$ at infinity. The moving sphere method is then applied to the resulting integral equation. In this process, we also need some new techniques to handle the absence of \eqref{eq 1.2} at infinity.

In the case $\sigma = 1$, the exponent condition $\alpha>-2$ in \eqref{Qeun2608} is sharp for the Liouville-type result in Theorem \ref{thm:no-asymptotic}. Indeed, Ni \cite{ni1982elliptic} constructed infinitely many positive solutions to \eqref{sobolev eq 7} under the decay assumption $|K(x)| \leq C |x|^{\beta}$ for $|x|$ large, with $\beta <-2$. When restricted to the standard sphere $(\mathbb{S}^{n}, g)$ (equivalently, condition \eqref{eq 1.2} holds), the assumption \eqref{Qeun2608} on $K$ can be removed. More precisely, we consider positive solutions of the $Q$-curvature equation
\begin{equation}\label{eq prescribing}
P_\sigma(u) = K(x) u^{\frac{n+2\sigma}{n-2\sigma}}\quad \text{on } \mathbb{S}^{n},
\end{equation}
where $P_\sigma$ is an intertwining operator. It can be viewed as the pull back operator of the higher order fractional Laplacian $(-\Delta)^\sigma$ in $\mathbb{R}^n$ via the stereographic projection under the assumption \eqref{eq 1.2} (see \cite{MR1316845,jin-nirenberg-2017}). Then we have the following result.

\begin{corollary}\label{thm sphere}
Let $K$ be continuous and rotationally symmetric on the standard sphere $(\mathbb{S}^{n}, g)$. In addition, assume that $K$ is monotone in the region where $K > 0$ and $K \not\equiv C$.
Then for every $0 < \sigma < n/2$, equation \eqref{eq prescribing} does not admit any positive solution $u \in C(\mathbb{S}^{n})$.
\end{corollary}

The paper is organized as follows. In Section \ref{S2}, we establish the equivalence between the higher order fractional equation \eqref{sobolev eq 7} and its corresponding integral equation, and then prove Theorem \ref{thm negative}. Sections \ref{radially} and \ref{non-rad} are devoted to proving Theorem \ref{thm radial} and Theorem \ref{thm Lin}, respectively. In Section \ref{sign-changing}, we show Theorem \ref{thm:no-asymptotic} and Corollary \ref{thm sphere}. In the appendix, we present some properties of hypergeometric functions and detailed calculations of the Pohozaev integrals applied in Theorem \ref{thm Lin}.

\section{Equivalence of the PDE and its integral representation}\label{S2}

In this section, our main goal is to establish the equivalence between \eqref{sobolev eq 7} and its integral equation. More generally, we will establish this equivalence for the following inhomogeneous equation
\begin{equation}\label{eq general}
    (-\Delta)^\sigma u  = f(x, u) \quad \text{in } \mathbb{R}^n.
\end{equation}
Our approach is inspired by \cite{ao2020removability,Y}, where the special case $f(x, u)=|x|^\alpha u^p$ was considered. We will extend this framework to general nonlinearities, including sign-changing $f(x, u)$. We also refer to Cao-Dai-Qin \cite{CDQ} on the integral representation for classical solutions of \eqref{eq general}, which is quite different from our assumptions on $f(x, u)$. We first recall the following lemma from Ao-Gonz\'alez-Hyder-Wei \cite[Lemma 5.2]{ao2020removability}.

\begin{lemma}[\cite{ao2020removability}]\label{lemma remo}
Let $\psi \in C^{\infty}\left(\mathbb{R}^n\right)$ be such that $\psi(x)=\frac{1}{|x|^\tau}$ on $B_1^c$ for some $\tau>0$. Let $\xi \in C^{\infty}(\mathbb{R}^n)$ be a cut-off function satisfying
\begin{equation*}
\xi(x)=1 \quad \text { for }|x| \leq 1 \quad \text { and } \quad \xi(x)=0 \quad\text { for }|x| \geq 2 .
\end{equation*}
For any $\varepsilon>0$, we define $\xi_{\varepsilon}:=\xi(\varepsilon x)$ and $\psi_{\varepsilon}(x):=\psi \xi_{\varepsilon}(x)$. Then for every $\sigma>0$ we have
\begin{equation*}
 (-\Delta)^\sigma \psi_{\varepsilon} \rightarrow(-\Delta)^\sigma \psi \quad \text { locally uniformly in } \mathbb{R}^n \text { as } \varepsilon \rightarrow 0.
\end{equation*}
Moreover, there exists $C=C(n, \sigma, \psi)>0$ (independent of $\varepsilon$) such that
\begin{equation*}
  \left|(-\Delta)^\sigma \psi_{\varepsilon}(x)\right|
  \leq C
  \begin{cases}
  (1+|x|)^{-2 \sigma-\tau} & \text { if } \tau<n, \\
  (1+|x|)^{-2 \sigma-\tau} \log (2+|x|) & \text { if } \tau=n, \\
  (1+|x|)^{-2 \sigma-n} & \text { if } \tau>n.
  \end{cases}
\end{equation*}

\end{lemma}

To proceed further, we impose the following structural assumption on the nonlinearity $f$.

\begin{assumption}\label{F}
     Let $f(x, t)$ be locally bounded in $x \in \mathbb{R}^n$ and $t>0$.
    Suppose there exist $R_0>0$, $C>0$, $\alpha >-2\sigma$ and $p>1$ such that
    \begin{equation*}
        |f(x, t)| \geq C |x|^\alpha t^p \quad \text{for } |x| \geq R_0 \text{ and } t\ge 0.
    \end{equation*}
\end{assumption}

\begin{proposition}\label{prop u lift}
Let $\sigma \in (0,  n/2)$ and $u \in \mathcal{L}_\sigma(\mathbb{R}^n)$ be a nonnegative distributional solution of
\eqref{eq general}. Suppose that $f$ satisfies Assumption \ref{F}, $f(\cdot, u) \in L_{\textmd{loc}}^1(\mathbb{R}^n)$, and there exists $R_1>0$ such that $f(x, u(x)) \geq 0$ or $f(x, u(x)) \leq 0$ in $\mathbb{R}^n \setminus B_{R_1}$.
\begin{enumerate}
    \item [(1)]  If $1 < p < \frac{n+\alpha}{n-2\sigma}$, then
    \begin{equation}\label{sobolev eq 2}
    \begin{aligned}
         & \int_{\mathbb{R}^n} \frac{|f(x,u(x))|}{1+|x|^\gamma}\, dx < +\infty
         \quad \text{for every } \gamma > 0, \\
         & \int_{\mathbb{R}^n} \frac{u(x)}{1+|x|^q}\, dx < +\infty
         \quad \text{for every } q > n - \frac{n+\alpha}{p}.
    \end{aligned}
    \end{equation}

    \item [(2)]  If $p \geq \frac{n+\alpha}{n-2\sigma}$, then
    \begin{equation}\label{sobolev eq 3}
    \begin{aligned}
         & \int_{\mathbb{R}^n} \frac{|f(x,u(x))|}{1+|x|^\gamma}\, dx < +\infty
         \quad \text{for every } \gamma > n-2\sigma-\frac{\alpha+2\sigma}{p-1}, \\
         & \int_{\mathbb{R}^n} \frac{u(x)}{1+|x|^q}\, dx < +\infty
         \quad \text{for every } q > n - \frac{\alpha+2\sigma}{p-1}.
    \end{aligned}
    \end{equation}
\end{enumerate}
\end{proposition}

\begin{proof}
For any $\tau>0$, take $\psi \in C^\infty(\mathbb{R}^n)$ such that
$\psi(x)=|x|^{-(n+\tau)}$ on $B_1^c$, and let $\psi_\varepsilon$ be defined as in Lemma~\ref{lemma remo}.
Using the dominated convergence theorem and the monotone convergence theorem, we obtain
\begin{equation*}
\begin{aligned}
\int_{\mathbb{R}^n} f(x,u(x)) \psi(x) \, dx
&= \lim_{\varepsilon \to 0} \int_{B_{R_1}} f(x,u(x)) \psi_\varepsilon(x) \, dx
   + \lim_{\varepsilon \to 0} \int_{B_{R_1}^c} f(x,u(x)) \psi_\varepsilon(x) \, dx \\
&= \lim_{\varepsilon \to 0} \int_{\mathbb{R}^n} u(x) (-\Delta)^\sigma \psi_\varepsilon(x) \, dx \\
&= \int_{\mathbb{R}^n} u(x)(-\Delta)^\sigma \psi(x) \, dx.
\end{aligned}
\end{equation*}
Hence
\begin{equation*}
\begin{aligned}
\Big|\int_{\mathbb{R}^n} f(x,u(x)) \psi(x) \, dx\Big|
&= \Big|\int_{\mathbb{R}^n} u(x)(-\Delta)^\sigma \psi(x) \, dx\Big| \\
&\leq \int_{\mathbb{R}^n} u(x) \, |(-\Delta)^\sigma \psi(x)| \, dx < +\infty,
\end{aligned}
\end{equation*}
where the last inequality follows from Lemma~\ref{lemma remo} and $u \in \mathcal{L}_\sigma(\mathbb{R}^n)$.
Therefore,
\begin{equation}\label{eq 1}
\int_{\mathbb{R}^n} \frac{|f(x,u(x))|}{1+|x|^{n+\tau}} \, dx < +\infty,
\quad \forall \, \tau>0.
\end{equation}
By H\"{o}lder's inequality and \eqref{eq 1}, for the constant $R_0$ given in Assumption \ref{F} and any $\tau>0$, we deduce that
\begin{equation*}
\begin{aligned}
\int_{\mathbb{R}^n} \frac{u(x)}{1+|x|^q} \, dx
&\leq C_1 + \int_{B_{R_0}^c} \frac{|x|^{\alpha/p} u(x)}{|x|^{(n+\tau)/p}}
                 \cdot \frac{1}{|x|^{q - (n+\tau-\alpha)/p}} \, dx \\
&\leq C_1 + \Bigg(\int_{B_{R_0}^c} \frac{|x|^\alpha u^p(x)}{|x|^{n+\tau}} \, dx\Bigg)^{1/p}
            \Bigg(\int_{B_{R_0}^c} \frac{dx}{|x|^{(q-(n+\tau-\alpha)/p)\frac{p}{p-1}}}\Bigg)^{1-1/p} \\
&{\leq} C_1 + \Bigg(\int_{B_{R_0}^c} \frac{|f(x,u(x))|}{|x|^{n+\tau}} \, dx\Bigg)^{1/p}
            \Bigg(\int_{B_{R_0}^c} \frac{dx}{|x|^{(qp+\alpha-n-\tau)/(p-1)}}\Bigg)^{1-1/p} \\
&< \infty, \quad  \forall\, q > n + \frac{\tau-\alpha}{p} > n - \frac{\alpha}{p}.
\end{aligned}
\end{equation*}
Thus, the integrability of $u$ lifts from $u \in \mathcal{L}_\sigma(\mathbb{R}^n)$ to $u \in \mathcal{L}_s(\mathbb{R}^n)$ for any $s > -\frac{\alpha}{2p}$.

Finally, applying the iteration method in \cite[Proposition~2.4]{Y}, we obtain the desired estimates (\ref{sobolev eq 2}) and (\ref{sobolev eq 3}).
\end{proof}

By Proposition~\ref{prop u lift}, the function
\begin{equation}\label{def v}
    v(x) := C_{n,\sigma} \int_{\mathbb{R}^n} \frac{f(y,u(y))}{|x-y|^{\,n-2\sigma}} \, dy
\end{equation}
is well-defined for every $x \in \mathbb{R}^n$,
where $C_{n,\sigma}$ is a positive constant such that
$C_{n,\sigma} |\cdot|^{\,2\sigma-n}$ is the fundamental solution of the fractional Laplacian.
We now establish the following growth estimates for $v$ at infinity.

\begin{proposition}\label{prop v}
Under the assumptions of Proposition~\ref{prop u lift}, let $v$ be defined by \eqref{def v}.
\begin{enumerate}
    \item If $p \geq \frac{n+\alpha}{n-2\sigma}$, then $v \in \mathcal{L}_s(\mathbb{R}^n)$ for any $s > -\frac{\alpha+2\sigma}{2(p-1)}$.
    \item If $1 < p < \frac{n+\alpha}{n-2\sigma}$, then $v \in \mathcal{L}_s(\mathbb{R}^n)$ for any $s > -\frac{n-2\sigma}{2}$.
\end{enumerate}
In particular, in both cases we have $v \in \mathcal{L}_0(\mathbb{R}^n)$.
\end{proposition}

\begin{proof}
    By Fubini's theorem, for any $q>2 \sigma$ we have
\begin{equation}\label{eq 4}
    \int_{\mathbb{R}^n} \frac{v(x)}{(1+|x|)^q} d x=C_{n, \sigma} \int_{\mathbb{R}^n}f(y,u(y))\left(\int_{\mathbb{R}^n} \frac{1}{|x-y|^{n-2 \sigma}} \frac{1}{(1+|x|)^q} d x\right) d y .
\end{equation}
If $|y| \leq 1$, then
\begin{equation*}
    \begin{aligned}
\int_{\mathbb{R}^n} \frac{1}{|x-y|^{n-2 \sigma}} \frac{1}{(1+|x|)^q} d x & \leq C\left(\int_{B_3} \frac{1}{|x|^{n-2 \sigma}} d x+\int_{B_2^c} \frac{1}{|x|^{n-2 \sigma+q}} d x\right) \\
& \leq C<\infty .
\end{aligned}
\end{equation*}
If $|y|>1$, then
\begin{equation*}
\int_{\mathbb{R}^n} \frac{1}{|x-y|^{n-2 \sigma}} \frac{1}{(1+|x|)^q} d x=: \sum_{i=1}^3 I_i
\end{equation*}
where
\begin{equation*}
\begin{aligned}
I_1 & =\int_{\left\{|x| \leq \frac{|y|}{2}\right\}} \frac{1}{|x-y|^{n-2 \sigma}} \frac{1}{(1+|x|)^q} d x \leq C \begin{cases}|y|^{2 \sigma-q}, & \text { if } q<n, \\
\log (1+|y|)|y|^{2 \sigma-n}, & \text { if } q=n, \\
|y|^{2 \sigma-n}, & \text { if } q>n,\end{cases} \\
I_2 & =\int_{\left\{\frac{|y|}{2}<|x|<2|y|\right\}} \frac{1}{|x-y|^{n-2 \sigma}} \frac{1}{(1+|x|)^q} d x \leq \frac{C}{|y|^q} \int_{\left\{\frac{|y|}{2}<|x|<2|y|\right\}} \frac{1}{|x-y|^{n-2 \sigma}} d x \\
& \leq C|y|^{2 \sigma-q}
\end{aligned}
\end{equation*}
and
\begin{equation*}
    I_3=\int_{\{|x| \geq 2|y|\}} \frac{1}{|x-y|^{n-2 \sigma}} \frac{1}{(1+|x|)^q} d x \leq C \int_{\{|x| \geq 2|y|\}} \frac{1}{|x|^{n-2 \sigma+q}} d x \leq C|y|^{2 \sigma-q} .
\end{equation*}
Let $u \in \mathcal{L}_\sigma(\mathbb{R}^n)$ be a nonnegative distributional solution of (\ref{eq general}) with $p \geq \frac{n+\alpha}{n-2 \sigma}$. Then for $q>n-\frac{\alpha+2 \sigma}{p-1}$ (note that $n-\frac{\alpha+2 \sigma}{p-1} \geq 2 \sigma$ due to $p \geq \frac{n+\alpha}{n-2 \sigma}$ ), by (\ref{eq 4}) and (\ref{sobolev eq 3}) we get
\begin{equation*}
    \int_{\mathbb{R}^n} \frac{|v(x)|}{(1+|x|)^q} d x \leq C \int_{B_1}|f(y,u(y))|d y+C \int_{B_1^c} \frac{\log (1+|y|)|f(y,u(y))|}{|y|^{q-2 \sigma}} d y<\infty .
\end{equation*}
That is, $v \in \mathcal{L}_s(\mathbb{R}^n)$ for any $s>-\frac{\alpha+2 \sigma}{2(p-1)}$.
Similarly, if $u \in \mathcal{L}_\sigma(\mathbb{R}^n)$ is a nonnegative distributional solution of (\ref{eq general}) in $\mathbb{R}^n$ with $1<p<\frac{n+\alpha}{n-2 \sigma}$, then by using (\ref{sobolev eq 2}) we get that $v \in \mathcal{L}_s(\mathbb{R}^n)$ for any $s>\frac{2 \sigma-n}{2}$.
Since $-\frac{\alpha+2 \sigma}{2(p-1)}<0$ and $\frac{2 \sigma-n}{2}<0$, we derive that $v\in\mathcal{L}_0(\mathbb{R}^n)$ in any case.
\end{proof}

Based on Propositions \ref{prop u lift} and \ref{prop v}, we establish the following integral representation of \eqref{eq general}. The converse is obvious, thus we obtain the equivalence between \eqref{eq general} and its integral equation.

\begin{theorem}\label{thm equivalence}
Let $\sigma \in (0,  n/2)$ and $u \in \mathcal{L}_\sigma(\mathbb{R}^n)$ be a nonnegative distributional solution of
\eqref{eq general}. Suppose that $f$ satisfies Assumption \ref{F}, $f(\cdot, u) \in L_{\textmd{loc}}^1(\mathbb{R}^n)$, and there exists $R_1>0$ such that $f(x, u(x)) \geq 0$ or $f(x, u(x)) \leq 0$ in $\mathbb{R}^n \setminus B_{R_1}$. Then $u$ also satisfies the integral equation
\begin{equation}\label{integral u}
    u(x) = C_{n,\sigma} \int_{\mathbb{R}^n}
    \frac{f(y,u(y))}{|x-y|^{\,n-2\sigma}} \, dy,
    \quad x \in \mathbb{R}^n,
\end{equation}
where $C_{n,\sigma}$ is a positive constant.
\end{theorem}

\begin{proof}
    Define $v$ as in \eqref{def v}. By Proposition~\ref{prop v}, we have $v \in \mathcal{L}_0(\mathbb{R}^n)$. Moreover, for any $\varphi \in C_c^{\infty}(\mathbb{R}^n)$,
    \begin{equation*}
    \begin{aligned}
        \int_{\mathbb{R}^n} v(x)(-\Delta)^\sigma \varphi(x)\,dx
        &= \int_{\mathbb{R}^n} \left( C_{n,\sigma} \int_{\mathbb{R}^n}
            \frac{f(y,u(y))}{|x-y|^{n-2\sigma}}\,dy \right)
            (-\Delta)^\sigma \varphi(x)\,dx \\
        &= \int_{\mathbb{R}^n} f(y,u(y)) \left( C_{n,\sigma}
            \int_{\mathbb{R}^n} \frac{(-\Delta)^\sigma \varphi(x)}
            {|x-y|^{n-2\sigma}}\,dx \right) dy \\
        &= \int_{\mathbb{R}^n} f(y,u(y)) \varphi(y)\,dy,
    \end{aligned}
    \end{equation*}
    where the second equality follows from Fubini's theorem. Thus $v$ is a distributional solution of
    \[
        (-\Delta)^\sigma v(x) = f(x,u(x)) \quad \text{in } \mathbb{R}^n .
    \]
Let $w := u-v$. Then $(-\Delta)^\sigma w = 0$ in $\mathbb{R}^n$ in the distributional sense. On the other hand, by Propositions~\ref{prop u lift} and~\ref{prop v}, both $u$ and $v$ belong to $\mathcal{L}_0(\mathbb{R}^n)$, hence $w \in \mathcal{L}_0(\mathbb{R}^n)$. By a Liouville type theorem we deduce that $w \equiv 0$ in $\mathbb{R}^n$. Therefore,
    \[
        u(x) = C_{n,\sigma} \int_{\mathbb{R}^n}
        \frac{f(y,u(y))}{|x-y|^{n-2\sigma}}\,dy
        \quad \text{for } x \in \mathbb{R}^n,
    \]
    which is the desired integral representation.
\end{proof}

Once Theorem \ref{thm equivalence} is established, Theorem \ref{thm negative} follows as a direct consequence.

\begin{proof}[Proof of Theorem \ref{thm negative}]
Let $f(x, t) = K(x)t^{\frac{n+2\sigma}{n-2\sigma}}$. Then the conditions of Theorem \ref{thm equivalence} are satisfied. It follows that $u$ also solves
\begin{equation*}
    u(x) = C_{n,\sigma} \int_{\mathbb{R}^n}
        \frac{K(y)u^\frac{n+2\sigma}{n-2\sigma}(y)}{|x-y|^{n-2\sigma}}\,dy
        \qquad \text{for } x \in \mathbb{R}^n.
\end{equation*}
Since $K(x)\le 0$, we obtain that $u(x)\le 0$. Thus $u \equiv 0$.
\end{proof}

\section{\texorpdfstring{Liouville theorem for radially symmetric $Q$-curvature}{Liouville theorem for radially symmetric $Q$-curvature}}\label{radially}

In this section, we will prove Theorem \ref{thm radial} by employing Theorem \ref{thm equivalence} along with the method of moving planes in integral form.

\begin{proof}[Proof of Theorem \ref{thm radial}] When $\sigma\in(0,1)$, Theorem \ref{thm radial} is a special case of Theorem \ref{thm Lin}. Our main focus here is on proving the case $\sigma \in [1, n/2)$. It follows directly from Theorem \ref{thm equivalence} that $u$ also satisfies the integral equation
\begin{equation*}
    u(x) = \int_{\mathbb{R}^n} \frac{K(|y|)u(y)^{\frac{n+2\sigma}{n-2\sigma}}}{|x-y|^{\,n-2\sigma}} \, dy.
\end{equation*}
Define
\begin{equation*}
v(x) := |x|^{-(n-2\sigma)} u \left(\frac{x}{|x|^2}\right), \qquad x \neq 0.
\end{equation*}

\begin{lemma}\label{lemma 4.1}
    Let $u$ and $K$ be given as in Theorem \ref{thm radial}. If $u$ is a positive solution, then  $v$ is radially symmetric and nonincreasing about the origin.
\end{lemma}

\begin{proof}
For $\lambda \in \mathbb{R}$, set
\begin{equation*}
x^\lambda := (2\lambda - x_1, x_2, \dots, x_n), \qquad v_\lambda(x) := v(x^\lambda),\qquad
\Sigma_\lambda := \{ x_1 \leq \lambda \} \setminus \{ 0^\lambda \}.
\end{equation*}
Let $\tau=\frac{n+2\sigma}{n-2\sigma}$. A direct computation gives
\begin{equation}\label{eq 4.2}
\begin{aligned}
v(x) - v_\lambda(x)
= \int_{\Sigma_\lambda}
\left( \frac{1}{|x-y|^{\,n-2\sigma}} - \frac{1}{|x^\lambda - y|^{\,n-2\sigma}} \right)
\Bigg(
K(\frac{y}{|y|^2}) v^\tau(y)
- K(\frac{y^\lambda}{|y^\lambda|^2}) v^\tau_\lambda(y)
\Bigg) dy .
\end{aligned}
\end{equation}

\noindent{\it Step 1. $v_\lambda \geq v$ in $\Sigma_\lambda$ for $\lambda$ sufficiently negative.}

Denote
\[
\Sigma_{\lambda}^{-} := \{ x \in \Sigma_{\lambda} \mid v(x) > v_{\lambda}(x) \}.
\]
For $y \in \Sigma_\lambda^-$, since $K$ is radially nondecreasing, we have
\begin{equation}\label{eq 4.3}
\begin{aligned}
    K(\frac{y}{|y|^2}) v^\tau(y)
- K(\frac{y^\lambda}{|y^\lambda|^2}) v_\lambda^\tau(y)
\leq K(\frac{y}{|y|^2}) \bigl(v^\tau(y)- v^\tau_\lambda(y)\bigr),
\quad  \forall\, y \in \Sigma_\lambda^-.
\end{aligned}
\end{equation}
For $y \in \Sigma_\lambda \setminus \Sigma_\lambda^-$, we have
\begin{equation}\label{eq 4.4}
K(\frac{y}{|y|^2}) v^\tau(y)
- K(\frac{y^\lambda}{|y^\lambda|^2}) v^\tau_\lambda(y) \leq K(\frac{y}{|y|^2}) \bigl(v^\tau(y)- v^\tau_\lambda(y)\bigr)
\leq 0, \quad  \forall\, y \in \Sigma_\lambda \setminus \Sigma_\lambda^- .
\end{equation}
Combining \eqref{eq 4.3} and \eqref{eq 4.4}, we deduce that
\begin{equation*}
\begin{aligned}
v(x) - v_\lambda(x)
&\leq \int_{\Sigma_\lambda^-} \left(\frac{1}{|x-y|^{n-2\sigma}} - \frac{1}{|x^\lambda - y|^{n-2\sigma}}\right)
K(\frac{y}{|y|^2}) \bigl(v^\tau(y)- v^\tau_\lambda(y)\bigr)\, dy\\
&\overset{MVT}{\leq} \tau \int_{\Sigma_\lambda^-} \left(\frac{1}{|x-y|^{n-2\sigma}} - \frac{1}{|x^\lambda - y|^{n-2\sigma}}\right)
K(\frac{y}{|y|^2}) v^{\tau-1} (y)\bigl(v(y) - v_\lambda(y)\bigr)\, dy\\
&\leq C \int_{\Sigma_\lambda^-} \frac{1}{|x-y|^{n-2\sigma}} v^{\tau-1}(y) \bigl(v(y) - v_\lambda(y)\bigr)\, dy,
\end{aligned}
\end{equation*}
where $C=C(n,\sigma,\|K\|_{L^\infty(B_1)})$.
By the Hardy-Littlewood-Sobolev inequality and H\"older inequality,
\begin{equation*}
\| v - v_\lambda \|_{L^q(\Sigma_\lambda^-)}
\leq C \left(\int_{\Sigma_\lambda^-} v^{\tau+1}\, dy \right)^{\frac{2\sigma}{n}}
\| v - v_\lambda \|_{L^q(\Sigma_\lambda^-)}, \qquad q > \frac{n}{n-2\sigma}.
\end{equation*}
Choosing $\lambda$ sufficiently negative such that
\begin{equation*}
C \left(\int_{\Sigma_\lambda^-} v^{\tau+1}\, dy \right)^{\frac{2\sigma}{n}} < \frac{1}{2},
\end{equation*}
we obtain
\begin{equation*}
\| v - v_\lambda \|_{L^q(\Sigma_\lambda^-)} = 0,
\end{equation*}
which implies $v_\lambda \geq v$ in $\Sigma_\lambda$.

Define
\begin{equation*}
\lambda_{0} = \sup \{ \lambda < 0 \mid v_{\mu} \geq v \text{ in } \Sigma_{\mu} \text{ for any } \mu < \lambda \}.
\end{equation*}

\noindent{\it Step 2. If $\lambda_{0} < 0$, then $v \equiv v_{\lambda_{0}}$ in $\Sigma_{\lambda_{0}}$.}

By definition, we have
\begin{equation*}
    v\le v_{\lambda_0} \quad\text{ in } \Sigma_{\lambda_0}.
\end{equation*}
We proceed by contradiction. Assume that $v \not\equiv v_{\lambda_{0}}$. We will show that under this assumption the plane can actually be moved further to the right.

Suppose $v \not\equiv v_{\lambda_{0}}$. Similar to \eqref{eq 4.4}, we have
\begin{equation*}
K(\frac{y}{|y|^{2}}) v^{\tau}(y)
- K(\frac{y^{\lambda_0}}{|y^{\lambda_0}|^2}) v^{\tau}_{\lambda_{0}}(y) \leq 0,
\quad y \in \Sigma_{\lambda_{0}}.
\end{equation*}
If $v(x_{0}) = v_{\lambda_{0}}(x_{0})$ for some $x_{0} \in \Sigma_{\lambda_{0}}$, then from \eqref{eq 4.2} we infer
\begin{equation*}
K(\frac{y}{|y|^{2}}) v^{\tau}(y)
- K(\frac{y^{\lambda_0}}{|y^{\lambda_0}|^2}) v^{\tau}_{\lambda_{0}}(y) \equiv 0,
\quad y \in \Sigma_{\lambda_{0}},
\end{equation*}
which implies $v \equiv v_{\lambda_{0}}$ in $\Sigma_{\lambda_{0}}$, a contradiction. Hence
\begin{equation}\label{eq 4.1}
v < v_{\lambda_{0}} \quad \text{in } \Sigma_{\lambda_{0}}.
\end{equation}
Once the strict inequality holds, one can extend the comparison beyond $\lambda_0$. More precisely, if \eqref{eq 4.1} holds, then there exists a small $\delta > 0$ such that $\lambda_{0} + \delta < 0$ and
\begin{equation}\label{eq 4.5}
v(x) \leq v_{\lambda}(x) \quad \text{in } \Sigma_{\lambda}, \qquad  \forall\, \lambda \in [\lambda_{0}, \lambda_{0}+\delta].
\end{equation}
Indeed, by similar calculations as in Step~1, we obtain
\begin{equation*}
\| v - v_{\lambda} \|_{L^{q}(\Sigma_{\lambda}^{-})}
\leq C \left( \int_{\Sigma_{\lambda}^{-}} v^{\tau+1} \, dy \right)^{\frac{2\sigma}{n}}
\| v - v_{\lambda} \|_{L^{q}(\Sigma_{\lambda}^{-})},
\end{equation*}
where $q > \frac{n}{n-2\sigma}$ and
\[
\Sigma_{\lambda}^{-} := \{ x \in \Sigma_{\lambda} \mid v(x) > v_{\lambda}(x) \}.
\]
It suffices to prove that $|\Sigma_{\lambda}^{-}| \to 0$ as $\delta \to 0$. Define
\begin{equation*}
E_{\varepsilon} := \{ x \in \Sigma_{\lambda_{0}} \mid v_{\lambda_{0}} - v > \varepsilon \},
\quad F_{\varepsilon} := \Sigma_{\lambda_{0}} \setminus E_{\varepsilon},
\quad H_{\lambda} := \Sigma_{\lambda} \setminus \Sigma_{\lambda_{0}}.
\end{equation*}
Clearly,
\begin{equation*}
\Sigma_{\lambda}^{-} \subset (\Sigma_{\lambda}^{-} \cap E_{\varepsilon}) \cup F_{\varepsilon}\cup H_{\lambda}.
\end{equation*}
We note that $\lim_{\lambda \to \lambda_{0}} |H_{\lambda}| = 0$, and since $v_{\lambda_{0}} > v$ in $\Sigma_{\lambda_{0}}$, we also have $\lim_{\varepsilon \to 0} |F_{\varepsilon}| = 0$. It remains to prove
\begin{equation*}
\lim_{\lambda \to \lambda_{0}} |\Sigma_{\lambda}^{-} \cap E_{\varepsilon}| = 0.
\end{equation*}
Indeed,
\begin{equation*}
v - v_{\lambda} = v - v_{\lambda_{0}} + v_{\lambda_{0}} - v_{\lambda} > 0
\quad \text{in } \Sigma_{\lambda}^{-} \cap E_{\varepsilon},
\end{equation*}
which implies
\[
v_{\lambda_{0}} - v_{\lambda} > v_{\lambda_{0}} - v \geq \varepsilon
\quad \text{in } \Sigma_{\lambda}^{-} \cap E_{\varepsilon}.
\]
Therefore
\begin{equation*}
E_{\varepsilon} \cap \Sigma_{\lambda}^{-}
\subset \{ x \in \Sigma_{\lambda} \mid v_{\lambda_{0}} - v_{\lambda} \geq \varepsilon \}
=: G_{\varepsilon}.
\end{equation*}
By Chebyshev's inequality,
\begin{equation*}
|G_{\varepsilon}| \leq \frac{1}{\varepsilon^{\tau+1}}
\int_{\Sigma_{\lambda}} |v_{\lambda_{0}} - v_{\lambda}|^{\tau+1}.
\end{equation*}
For fixed $\varepsilon > 0$, as $\lambda \to \lambda_{0}$ we have $|G_{\varepsilon}| \to 0$. Thus we can fix a small $\varepsilon>0$ and then choose $\delta > 0$ sufficiently small such that
\begin{equation*}
C \left( \int_{\Sigma_{\lambda}^{-}} v^{\tau+1} \, dy \right)^{\frac{2\sigma}{n}} < \frac{1}{2},
\qquad  \forall\, \lambda \in [\lambda_{0}, \lambda_{0}+\delta].
\end{equation*}
Consequently, $\Sigma_{\lambda}^{-} = \emptyset$, and \eqref{eq 4.5} holds. However, \eqref{eq 4.5} contradicts the definition of $\lambda_{0}$. Hence \eqref{eq 4.1} cannot hold, and we conclude that
\[
v \equiv v_{\lambda_{0}} \quad \text{in } \Sigma_{\lambda_{0}}, \qquad \text{if } \lambda_{0} < 0.
\]

\noindent{\it Step 3. Classification of $u$.}

{Case 1.} If $\lambda_{0}<0$, then $v \equiv v_{\lambda_{0}}$ in $\Sigma_{\lambda_{0}}$. In particular, $v$ is continuous at $0$, which means $u$ decays at infinity at the rate $|x|^{-(n-2\sigma)}$. By a similar calculation as in Step~2, we have
\begin{equation*}
K(\frac{y}{|y|^{2}})
\equiv
 K(\frac{y^{\lambda_0}}{|y^{\lambda_0}|^2}),
\quad y \in \Sigma_{\lambda_{0}}.
\end{equation*}
Thus $K$ must be constant, contradicting the assumption.

{Case 2.} If $\lambda_{0}=0$, then we may move the plane from the other side and obtain the corresponding $\lambda_{0}'$.

\begin{enumerate}

  \item If $\lambda_{0}'>0$, the argument is analogous to Case~1, leading to a contradiction.
  \item If $\lambda_{0}'=0$, then $v(x) \equiv v(x^{0})$ in $\Sigma_{0}$, which shows that $v(x)$ is symmetric with respect to $\{x_{1}=0\}$.
  Since the $x_{1}$-axis can be chosen in any direction, $v(x)$ must be radially symmetric and radially decreasing about the origin.

\end{enumerate}
This completes the proof of Lemma \ref{lemma 4.1}.
\end{proof}

Since equation \eqref{sobolev eq 7} is not translation-invariant, we cannot apply the Kelvin transform centered at an arbitrary point to classify the solutions. To overcome this difficulty, we make use of the following result from Bianchi \cite{gabriele_nonexistence_1997}, and then select an appropriate point to perform the Kelvin transform.

\begin{lemma}[\cite{gabriele_nonexistence_1997}]\label{lemma bianchi}
Let $u$ be a function defined on $\mathbb{R}^n$. Suppose that $u$ is radially symmetric about some point $P$ and also its Kelvin transform centered in a point different from $P$ is radially symmetric about some point $Q$.
Then either $u$ is constant or
\begin{equation*}
u(x) = \frac{k}{\left( |x-x_{0}|^{2}+h \right)^{\frac{n-2\sigma}{2}}}
\end{equation*}
for some constants $k,h >0$ and some point $x_{0} \in \mathbb{R}^{n}$.
\end{lemma}

\begin{lemma}\label{Lemma 4.3}
    Let $u$ and $K$ be given as in Theorem \ref{thm radial}. Assume $u$ is a positive solution. If $v$ is radially symmetric and decreasing about the origin, then $v$ satisfies the assumptions of Lemma \ref{lemma bianchi}.
\end{lemma}

\begin{proof}
We will adapt some ideas from Bianchi \cite{gabriele_nonexistence_1997}. By Lemma \ref{lemma bianchi}, we only need to verify that the Kelvin transform of $v$, when centered at a point different from the origin, is itself radially symmetric about some point. Let $T=(0,\cdots,0,T_n)$ with $T_n<0$, $H=T-\frac{T}{|T|^2}$. Define
\begin{equation*}
w(x) = \begin{cases}
   \frac{1}{|x-T|^{\,n-2\sigma}}
v\!\left( \frac{x-T}{|x-T|^{2}} + T \right),
\qquad &x \neq T,\\
u(0),\qquad &x=T.
\end{cases}
\end{equation*}
Then $w$ is axisymmetric with respect to the $x_{n}$-axis.
To prove that $w$ is {radially symmetric}, it suffices to show that $w$ is also symmetric with respect to some hyperplane whose normal $\theta$ forms an angle with the $x_{n}$-axis that is an irrational multiple of $\pi$.
We will use the moving plane method once again.

Let
\[
{z}= T + \frac{x-T}{|x-T|^{2}},
\]
Then
\begin{equation*}
|x| = \frac{|T||z-H|}{|{z} - T|}.
\end{equation*}
Since $v(x) = v(|x|)$, we obtain
\begin{equation*}
w(z) = \begin{cases}
   \frac{1}{|z-T|^{\,n-2\sigma}}
v\!\left( \frac{|T||z-H|}{|z-T|} \right),
\qquad &z \neq T,H.\\
u(0),\qquad &z=T.
\end{cases}
\end{equation*}
Without loss of generality, assume $\theta$ points towards $\{z_n > 0\}$. Denote $z^\lambda := z + 2(\lambda - z\cdot \theta)\theta$ as the reflection of $z$
with respect to the hyperplane $z\cdot \theta = \lambda$, and $w_\lambda(z)=w(z^\lambda)$. Let
\[
\Sigma_\lambda = \{ z : z\cdot \theta \leq \lambda \}\backslash\{H^\lambda\}.
\]
We will move the plane $z\cdot \theta = \lambda$ from $\lambda=T\cdot\theta$ towards $\lambda=H\cdot\theta$.
\begin{figure}[htbp]
    \centering
    \includegraphics[width=0.5\linewidth]{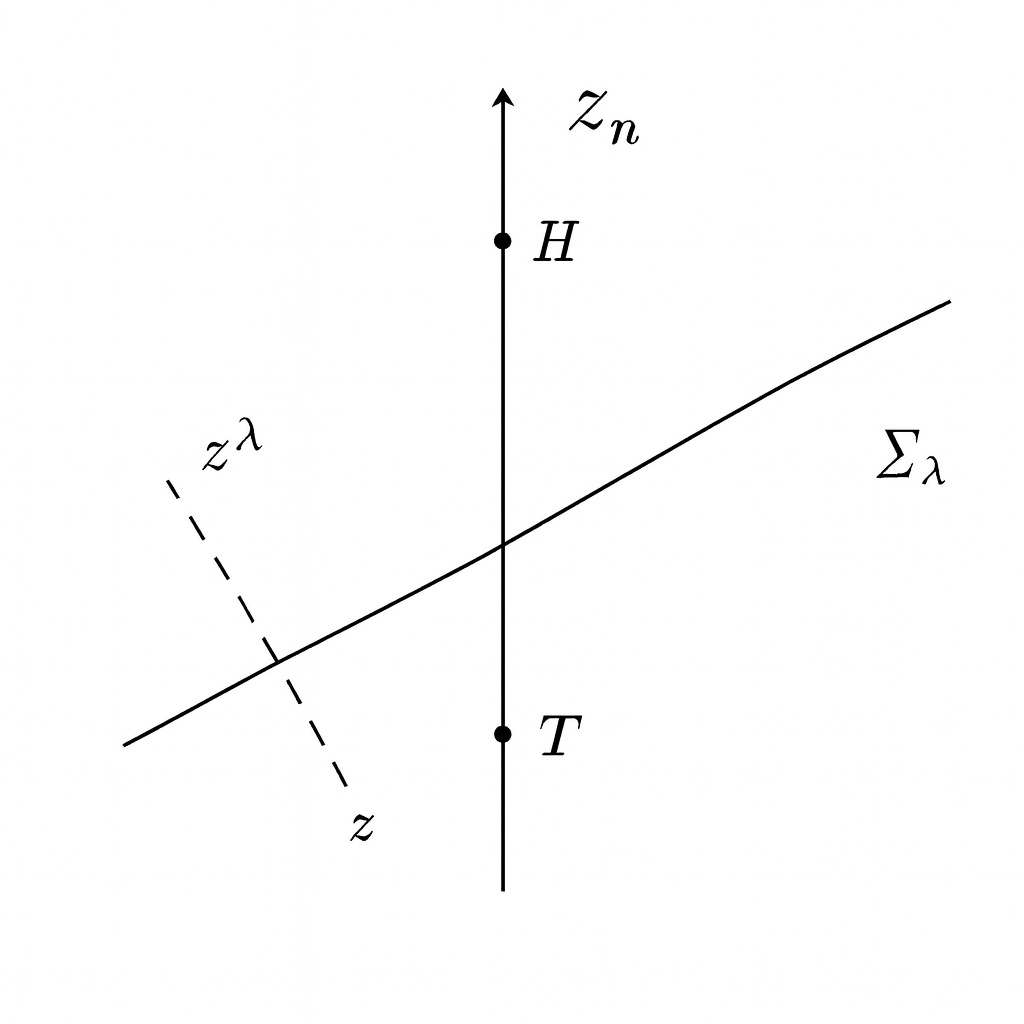}
    \caption{$T$, $H$ and $\Sigma_\lambda$.}
    \label{fig:placeholder}
\end{figure}

\noindent{\it Step 1. $w(z) \leq w(z^\lambda)$ in $\Sigma_\lambda$ when $\lambda=T\cdot\theta$.}

For $\lambda=T\cdot\theta$ and $z\in\Sigma_\lambda$, we have
\begin{equation*}
    |z-T|=|z^\lambda-T|,\qquad |z-H|>|z^\lambda-H|.
\end{equation*}
Since $v(|x|)$ is nonincreasing about $|x|$, we have
\begin{equation*}
    w(z) =   \frac{1}{|z-T|^{\,n-2\sigma}}
v\!\left( \frac{|T||z-H|}{|z-T|} \right)\le \frac{1}{|z^\lambda-T|^{\,n-2\sigma}}
v\!\left( \frac{|T||z^\lambda-H|}{|z^\lambda-T|} \right) =w(z^\lambda).
\end{equation*}
Define
\begin{equation*}
\lambda_{0} := \sup \{ T\cdot\theta\le\lambda < H\cdot\theta    \mid w(x) \leq w(x^{\mu}) \text{ in } \Sigma_{\mu} \text{ for all } T\cdot\theta\le \mu < \lambda \}.
\end{equation*}

\noindent{\it Step 2. $T\cdot\theta\le\lambda_{0}< H\cdot\theta$.}

If $\sigma = 1$, we have
\begin{equation*}
-\Delta u(x) = K(x)u^\frac{n+2\sigma}{n-2\sigma}(x) \ge 0.
\end{equation*}
If $\sigma > 1$, it follows from the dominated convergence theorem that
\begin{equation*}
-\Delta u(x) = (n - 2\sigma)(2\sigma - 2)
\int_{\mathbb{R}^n} \frac{K(y)u^\frac{n+2\sigma}{n-2\sigma}(y)}{|x - y|^{\,n - 2\sigma + 2}} \, dy \ge 0.
\end{equation*}
Since $u(x) = u(|x|)$ is radially symmetric, we obtain
\begin{equation*}
-\frac{1}{r^{n-1}} \left( r^{n-1} u'(r) \right)' \ge 0.
\end{equation*}
Let $h(r) = r^{n-1} u'(r)$. Then $h(0) = 0$ and $h'(r) \le 0$ for all $r \ge 0$.
Consequently, either $h(r) \equiv 0$ for all $r \ge 0$, or there exists some $r_0 > 0$ such that $h(r_0) < 0$.
If $h(r) \equiv 0$, then $u \equiv C$, which means $K\equiv0$ and contradicts the assumption. Thus $h(r_0) < 0$ for some $r_0 > 0$. In this case, $u'(r_0) < 0$, which implies
\begin{equation*}
\left. \frac{d}{dr} \left( \frac{1}{r^{n-2\sigma}} v(1/r) \right) \right|_{r = r_0} < 0.
\end{equation*}
Hence,
\begin{equation*}
-(n - 2\sigma) v(r_0^{-1}) - r_0^{-1}\, v'(r_0^{-1}) < 0.
\end{equation*}
On the other hand, one can verify that
\begin{equation*}
\lim_{\substack{z \cdot \theta = \lambda \\ |z| \to \infty}}
|z - T|^{\,n - 2\sigma + 2} \frac{\partial w}{\partial \theta}
= -(n - 2\sigma)(\lambda - T \cdot \theta) v(|T|)
- (H \cdot \theta - T \cdot \theta)\, |T|\, v'(|T|),
\end{equation*}
Choose $T$ such that $|T|=r_0^{-1}$, then
\begin{equation*}
-(n - 2\sigma) v(|T|) - |T|\, v'(|T|) < 0.
\end{equation*}
Then, for $\lambda < H \cdot \theta$ and $\lambda$ sufficiently close to $H \cdot \theta$, we have
\begin{equation*}
\lim_{\substack{z \cdot \theta = \lambda \\ |z| \to \infty}}
|z - T|^{\,n - 2\sigma + 2}  \frac{\partial w}{\partial \theta} < 0.
\end{equation*}
This means
\begin{equation*}
\frac{\partial w}{\partial \theta}(z) < 0
\quad \text{for some } z \in \{ z \cdot \theta = \lambda \}.
\end{equation*}
Therefore,
\begin{equation*}
w(z^\lambda) < w(z)
\quad \text{for some } z \in \{ z \cdot \theta < \lambda \}.
\end{equation*}
It follows that $\lambda_0$ cannot reach $H \cdot \theta$; hence $\lambda_0 < H \cdot \theta$.

\vspace{0.1cm}
\noindent{\it Step 3. $w(z) \equiv w(z^{\lambda_{0}})$ in ${\Sigma_{\lambda_{0}}}$.}

We argue by contradiction. Assume $w \not\equiv w_{\lambda_{0}}$.
Similar to the proof of Lemma \ref{lemma 4.1}, we have
\begin{equation*}
w < w_{\lambda_{0}} \quad \text{in } \Sigma_{\lambda_{0}}.
\end{equation*}
If this strict inequality holds, then we can move the plane further to the right:
there exists a small $\delta > 0$ such that $\lambda_{0} + \delta < H \cdot \theta$ and
\begin{equation}\label{eq 4.8}
w(x) \leq w_{\lambda}(x) \quad \text{in } {\Sigma_{\lambda}},
\qquad  \forall\, \lambda \in [\lambda_{0}, \lambda_{0}+\delta].
\end{equation}
Next we verify \eqref{eq 4.8} to get the contradiction. After some calculations, we arrive at
\begin{equation}\label{eq 3.3}
\begin{aligned}
&w(x) - w(x^\lambda)
= \int_{\Sigma_\lambda} \!\left( \frac{1}{|x-z|^{\,n-2\sigma}} - \frac{1}{|x^\lambda - z|^{\,n-2\sigma}} \right) \\
&\quad \times \Bigg[
K(\frac{|z-T|}{|T||z-H|})  w^\tau(z)  -
K(\frac{|z^\lambda-T|}{|z^\lambda-H||T|}) w^\tau(z^\lambda)
\Bigg] dz.
\end{aligned}
\end{equation}
Moreover, one can verify that
\begin{equation*}
\frac{1}{|x-z|^{\,n-2\sigma}} - \frac{1}{|x^\lambda - z|^{\,n-2\sigma}} > 0,
\qquad \text{for } x,z \in \Sigma_\lambda.
\end{equation*}
For $\lambda \in [\lambda_{0}, \lambda_{0}+\delta]$ and $z \in \Sigma_\lambda$, we have
\begin{equation*}
|z-T| <|z^\lambda - T|,
\qquad |z-H| > |z^\lambda -H|.
\end{equation*}
Define
\begin{equation*}
\Sigma_\lambda^{-} := \{ x \in \Sigma_\lambda \mid w(x) >w(x^\lambda)\}.
\end{equation*}
For $z \in \Sigma_\lambda^{-}$,
\begin{equation}\label{eq 3.1}
\begin{aligned}
   K(\frac{|z-T|}{|T||z-H|}) w^{\tau}(z)- K(\frac{|z^{\lambda}-T|}{|T||z^\lambda-H|}) w^{\tau}(z^{\lambda})
\leq  K(\frac{|z-T|}{|T||z-H|})\bigl(w^{\tau}(z) - w^{\tau}(z^{\lambda})\bigr).
\end{aligned}
\end{equation}
For $z \in \Sigma_{\lambda} \setminus {\Sigma_{\lambda}^{-}}$, we have
\begin{equation}\label{eq 3.2}
\begin{aligned}
 K(\frac{|z-T|}{|T||z-H|}) w^{\tau}(z)
\leq  K(\frac{|z^{\lambda}-T|}{|T||z^\lambda-H|}) w^{\tau}(z^{\lambda}).
\end{aligned}
\end{equation}
Combining \eqref{eq 3.3}-\eqref{eq 3.2}, we obtain
\begin{equation*}
\begin{aligned}
& w(x) - w(x^\lambda) \\
&\leq \int_{\Sigma_{\lambda}^{-}} \!\left( \frac{1}{|x-z|^{n-2\sigma}} - \frac{1}{|x^{\lambda}-z|^{n-2\sigma}} \right)
K(\frac{|z-T|}{|z-H||T|})
\bigl(w^\tau(z) - w^\tau(z^{\lambda})\bigr)\,dz.
\end{aligned}
\end{equation*}
By the mean value theorem,
\begin{equation*}
\begin{aligned}
    w(x) - w(x^{\lambda})
&\leq \tau \!\!\int_{\Sigma_{\lambda}^{-}} \!\frac{1}{|x-z|^{n-2\sigma}}
K(\frac{|z-T|}{|z-H||T|})
w^{\tau-1}(z)\bigl(w(z) - w(z^{\lambda})\bigr)\,dz.
\end{aligned}
\end{equation*}
Notice that
\begin{equation*}
\frac{|z - T|}{|z - H| \, |T|}
\le \frac{1}{|T|} \left( 1 + \frac{|H - T|}{|z - H|} \right)
= \frac{1}{|T|} \left( 1 + \frac{1}{|T| \, |z - H|} \right)
\le C(T, \lambda).
\end{equation*}
Then
\begin{equation*}
    \begin{aligned}
  w(x) - w(x^{\lambda})&\leq C(n,\sigma,T,\lambda) \!\int_{\Sigma_{\lambda}^{-}} \frac{1}{|x-z|^{n-2\sigma}} w^{\tau-1}(z)\bigl(w(z) - w(z^{\lambda})\bigr)\,dz.
    \end{aligned}
\end{equation*}
By the Hardy-Littlewood-Sobolev inequality and H\"older inequality,
\begin{equation*}
\|w - w_{\lambda}\|_{L^{q}(\Sigma_{\lambda}^{-})}
\leq C(n,\sigma) \left(\int_{\Sigma_{\lambda}^{-}} w^{\tau+1}(z)\,dz\right)^{\frac{2\sigma}{n}}
\|w - w_{\lambda}\|_{L^{q}(\Sigma_{\lambda}^{-})}, \qquad q > \frac{n}{n-2\sigma}.
\end{equation*}
It suffices to prove that $|\Sigma_{\lambda}^{-}| \to 0$ as $\delta \to 0$. For $\lambda\in[\lambda_0,\lambda_0+\delta]$, let
\begin{equation*}
E_{\varepsilon} := \{ x \in \Sigma_{\lambda_{0}} \mid w_{\lambda_{0}} - w > \varepsilon \}, \quad
F_{\varepsilon} := \Sigma_{\lambda_{0}} \setminus E_{\varepsilon}, \quad
H_{\lambda} := \Sigma_{\lambda} \setminus \Sigma_{\lambda_{0}}.
\end{equation*}
Then
\begin{equation*}
\Sigma_{\lambda}^{-} \subset (\Sigma_{\lambda}^{-} \cap E_{\varepsilon})
\cup F_{\varepsilon}
\cup H_{\lambda}.
\end{equation*}
Clearly,
\[
\lim_{\lambda \downarrow \lambda_{0}} |H_{\lambda}| = 0,
\qquad
\lim_{\varepsilon \to 0} |F_{\varepsilon}| = 0
\quad (\text{since } w \leq w_{\lambda_{0}} \text{ in } \Sigma_{\lambda_{0}}).
\]
Thus we only need to show
\begin{equation*}
\lim_{\lambda \downarrow \lambda_{0}} |\Sigma_{\lambda}^{-} \cap E_{\varepsilon}| = 0.
\end{equation*}
On $\Sigma_{\lambda}^{-} \cap E_{\varepsilon}$ we have
\[
w - w_{\lambda} = w- w_{\lambda_{0}} + w_{\lambda_{0}} - w_{\lambda} > 0,
\]
and therefore
\[
w_{\lambda_{0}} - w_{\lambda} > w_{\lambda_{0}} - w > \varepsilon.
\]
Hence
\begin{equation*}
E_{\varepsilon} \cap \Sigma_{\lambda}^{-}
\subset \{ x \in \Sigma_{\lambda} \mid w_{\lambda_{0}} - w_{\lambda} > \varepsilon \}
=: G_{\varepsilon}.
\end{equation*}
By Chebyshev's inequality,
\begin{equation*}
|G_{\varepsilon}| \leq \frac{1}{\varepsilon^{\tau+1}}
\int_{\Sigma_{\lambda}} |w_{\lambda} - w_{\lambda_{0}}|^{\tau+1} .
\end{equation*}
For fixed $\varepsilon > 0$, letting $\lambda \downarrow \lambda_{0}$ gives $|G_{\varepsilon}| \to 0$.
Thus we can choose $\delta > 0$ sufficiently small so that
\begin{equation*}
C \!\left( \int_{\Sigma_{\lambda}^{-}} w^{\tau+1} dy \right)^{\!\frac{2\sigma}{n}} < \frac{1}{2},
\qquad  \forall\, \lambda \in [\lambda_{0}, \lambda_{0}+\delta].
\end{equation*}
Consequently $\Sigma_{\lambda}^{-} =\emptyset$ and hence $w \leq w_{\lambda}$ for $\lambda \in [\lambda_{0}, \lambda_{0}+\delta]$. Thus \eqref{eq 4.8} is verified, which contradicts the definition of $\lambda_{0}$.

Therefore,
\[
w \equiv w_{\lambda_{0}} \quad \text{in } \Sigma_{\lambda_{0}}.
\]
This implies $w$ is symmetric about the hyperplane $z\cdot \theta=\lambda_0$. This completes the proof of Lemma \ref{Lemma 4.3}.
\end{proof}

Lemma \ref{Lemma 4.3} implies that $v$ satisfies the assumptions of Lemma \ref{lemma bianchi}.
Hence, either $v$ is constant, or
\begin{equation*}
v(x) = \frac{k}{\left( |x - x_{0}|^{2} + h \right)^{\frac{n - 2\sigma}{2}}}
\end{equation*}
for some constants $k, h > 0$ and some point $x_{0} \in \mathbb{R}^{n}$. This further implies that $K$ is a constant, contradicting the assumption. Thus, the proof of Theorem \ref{thm radial} is finished.
\end{proof}

\section{\texorpdfstring{Liouville theorem for non-radially symmetric $Q$-curvature}{Liouville theorem for non-radially symmetric $Q$-curvature}}\label{non-rad}

In this section, we prove Theorem \ref{thm Lin} by combining the method of moving spheres in integral form, the technique of local blow-up analysis, and the Pohozaev identity.

\begin{proof}[Proof of Theorem \ref{thm Lin}] We argue by contradiction. Suppose that $u$ is a positive solution. It follows directly from Theorem \ref{thm equivalence} that $u$ also satisfies the integral equation
\begin{equation*}
    u(x) = \int_{\mathbb{R}^n} \frac{K(y)\, u^\frac{n+2\sigma}{n-2\sigma}(y)}{|x-y|^{\,n-2\sigma}} \, dy.
\end{equation*}
Denote $\tau=\frac{n+2\sigma}{n-2\sigma}$. For $\lambda > 0$, we define the Kelvin transform of $u$ by
\begin{equation*}
    u_\lambda(x) = \left( \frac{\lambda}{|x|} \right)^{n-2\sigma}
    u\!\left( \frac{\lambda^2 x}{|x|^2} \right),\quad x\neq 0.
\end{equation*}
A direct calculation shows that
\begin{equation}\label{sobolev eq 8}
\begin{aligned}
  u(x)-u_\lambda(x)
  &= \int_{B_\lambda}
  \Bigg( \frac{1}{|x-y|^{\,n-2\sigma}}
  - \frac{1}{\Big|\frac{\lambda^2 x}{|x|^2}-y\Big|^{\,n-2\sigma}}
    \left(\frac{\lambda}{|x|}\right)^{n-2\sigma} \Bigg)  \\
  &\qquad\quad \times
   \left( K(y)\, u^\tau(y) - K \! \left(\frac{\lambda^2 y}{|y|^2}\right) u_\lambda^\tau(y) \right) \, dy.
\end{aligned}
\end{equation}
Moreover, one can verify that
\[
   \frac{1}{|x-y|^{\,n-2\sigma}}
   - \frac{1}{\Big|\frac{\lambda^2 x}{|x|^2}-y\Big|^{\,n-2\sigma}}
     \left(\frac{\lambda}{|x|}\right)^{n-2\sigma} > 0,
   \qquad x,y \in B_\lambda \setminus \{0\}.
\]

\begin{lemma}\label{lemma 3.1}
Assume that $n\ge 2$ and $0< \sigma< n/2$. Let $u$ and $K$ be given as in Theorem \ref{thm Lin}. If $u$ is a positive solution, then
    $|x|^{\frac{n-2\sigma}{2}} u(x)$ is increasing along each ray $\{\, t \xi : t \geq 0 \,\}$ where $\xi\in\mathbb{R}^n$ and $|\xi| = 1$.
\end{lemma}

\begin{proof}

\noindent{\it Step 1.  $u_\lambda \geq u$ in $B_\lambda \setminus \{0\}$ for $\lambda > 0$ sufficiently small.}

Define
\[
\Sigma_\lambda^{-} = \{\, x \in B_\lambda \setminus \{0\} \;\mid\; u_\lambda(x) < u(x) \,\}.
\]
If $y \in \ B_\lambda \setminus \{0\}$, then
\begin{equation*}
    K(y) u^\tau(y) - K\!\left(\frac{\lambda^2 y}{|y|^2}\right) u_\lambda^\tau(y)
    \leq K(y)\,\big(u^\tau(y)-u_\lambda^\tau(y)\big).
\end{equation*}
Moreover, if $y \in B_\lambda \setminus (\Sigma_\lambda^{-} \cup \{0\})$, we have
\begin{equation*}
    K(y) u^\tau(y) - K\!\left(\frac{\lambda^2 y}{|y|^2}\right) u_\lambda^\tau(y)
    \leq K(y)\,\big(u^\tau(y)-u_\lambda^\tau(y)\big) \leq 0.
\end{equation*}
Therefore, for any $x \in B_\lambda \setminus \{0\}$,
\begin{equation*}
\begin{aligned}
u(x)-u_\lambda(x)
&\leq \int_{\Sigma_\lambda^{-}}
\bigg( \frac{1}{|x-y|^{\,n-2\sigma}}
 - \frac{1}{\big|\frac{\lambda^2 x}{|x|^2}-y\big|^{\,n-2\sigma}}
   \left(\frac{\lambda}{|x|}\right)^{n-2\sigma} \bigg) \\
&\qquad \times K(y)\,\big(u^\tau(y)-u_\lambda^\tau(y)\big)\,dy \\[4pt]
&= \tau \int_{\Sigma_\lambda^{-}}
\bigg( \frac{1}{|x-y|^{\,n-2\sigma}}
 - \frac{1}{\big|\frac{\lambda^2 x}{|x|^2}-y\big|^{\,n-2\sigma}}
   \left(\frac{\lambda}{|x|}\right)^{n-2\sigma} \bigg) \\
&\qquad \times K(y)\, \xi^{\tau-1}\,\big(u(y)-u_\lambda(y)\big)\,dy \\[4pt]
&\leq \tau \int_{\Sigma_\lambda^{-}} \frac{1}{|x-y|^{\,n-2\sigma}} K(y)\,u^{\tau-1}(y)\,\big(u(y)-u_\lambda(y)\big)\,dy \\
&\leq C \int_{\Sigma_\lambda^{-}} \frac{1}{|x-y|^{\,n-2\sigma}} u^{\tau-1}(y)\,\big(u(y)-u_\lambda(y)\big)\,dy,
\end{aligned}
\end{equation*}
where $u_\lambda(y) < \xi < u(y)$ and $C=C(n,\sigma,\|K\|_{L^\infty(B_1)})$.

By the Hardy-Littlewood-Sobolev inequality and H\"older inequality, we deduce for any $q > \frac{n}{n-2\sigma}$ that
\begin{equation}\label{eq:HLS-step}
    \| u - u_\lambda \|_{L^q(\Sigma_\lambda^-)}
    \leq C \left( \int_{\Sigma_\lambda^-} u^{\tau+1}(y)\,dy \right)^{\frac{2\sigma}{n}}
       \| u - u_\lambda \|_{L^q(\Sigma_\lambda^-)}.
\end{equation}
Since $u \in C(\mathbb{R}^n)$, the integral in \eqref{eq:HLS-step} can be made arbitrarily small when $\lambda>0$ is chosen sufficiently small. In particular, we obtain
\[
\| u - u_\lambda \|_{L^q(\Sigma_\lambda^-)}
\leq \frac{1}{2}\, \| u - u_\lambda \|_{L^q(\Sigma_\lambda^-)}.
\]
Hence, it follows that
\begin{equation*}
    u_\lambda \geq u \quad \text{in } B_\lambda \setminus \{0\}.
\end{equation*}
Step 1 is verified.

Define
\begin{equation}
\begin{aligned}
\lambda_0
   &= \sup \left\{ \lambda > 0 \;\middle|\;
      u \leq u_\rho \;\; \text{in } B_\rho \setminus \{0\}, \;\;  \forall\, 0 < \rho < \lambda \right\}.
\end{aligned}
\end{equation}

\noindent{\it Step 2. $\lambda_0 = +\infty$.}

Suppose, for contradiction, that $\lambda_0 < +\infty$.
By the definition of $\lambda_0$, we have
\begin{equation*}
u \leq u_{\lambda_0} \quad \text{in } B_{\lambda_0}\setminus \{0\}.
\end{equation*}
One can show that the inequality is strict.
Indeed, since
\begin{equation*}
K(y) u^{\tau}(y) - K\!\left(\frac{\lambda_{0}^{2} y}{|y|^{2}}\right)
u_{\lambda_{0}}^{\tau}(y)
\leq K(y)\bigl(u^{\tau}(y) - u_{\lambda_{0}}^{\tau}(y)\bigr) \leq 0,
\qquad y \in B_{\lambda_{0}} \setminus \{0\},
\end{equation*}
if there exists \(x_{0} \in B_{\lambda_{0}} \setminus \{0\}\) such that
\(u(x_{0}) = u_{\lambda_{0}}(x_{0})\),
then by \eqref{sobolev eq 8} it follows that
\begin{equation*}
K(y) u^\tau(y) \equiv K\!\left(\frac{\lambda_0^2 y}{|y|^2}\right) u_{\lambda_0}^\tau(y),
   \qquad y \in B_{\lambda_0}\setminus \{0\}.
\end{equation*}
By \eqref{sobolev eq 8} again, we have
\begin{equation*}
u \equiv u_{\lambda_0} \quad \text{in } B_{\lambda_0}\setminus \{0\},
\end{equation*}
which further implies that
\begin{equation*}
K(y) \equiv K\!\left(\frac{\lambda_0^2 y}{|y|^2}\right),
   \qquad y \in B_{\lambda_0}\setminus \{0\}.
\end{equation*}
Thus $K$ must be constant, contradicting the assumption.
Hence
\begin{equation*}
u < u_{\lambda_0} \quad \text{in } B_{\lambda_0}\setminus \{0\}.
\end{equation*}
Once the strict inequality is known, one can extend the comparison beyond $\lambda_0$.
More precisely, there exists a small $\delta > 0$ such that
\begin{equation*}
u(x) \leq u_\lambda(x)
\quad \text{for } x \in B_\lambda \setminus \{0\},
\quad \lambda \in [\lambda_0, \lambda_0 + \delta].
\end{equation*}
Indeed, for $x \in B_\lambda \setminus \{0\}$ and $\lambda \in [\lambda_0, \lambda_0 + \delta]$,
set
\begin{equation*}
\Sigma_\lambda^- := \{\, x \in B_\lambda \setminus \{0\} : u(x) > u_\lambda(x) \,\}.
\end{equation*}
Using an argument similar to that in Step~1, one checks that for any $q > \frac{n}{n-2\sigma}$,
\begin{equation*}
\|u - u_\lambda\|_{L^q(\Sigma_\lambda^-)}
\leq C \left( \int_{\Sigma_\lambda^-} u^{\tau+1}(y)\, dy \right)^{\frac{2\sigma}{n}}
   \|u - u_\lambda\|_{L^q(\Sigma_\lambda^-)}.
\end{equation*}
For $\delta>0$ sufficiently small, we can verify that
\(|\Sigma_{\lambda}^{-}|\) is arbitrarily small.
Indeed, for any $\varepsilon>0$, define
\begin{equation*}
E_{\varepsilon} =
\left\{ x \in B_{\lambda_{0}}\setminus\{0\} \ \middle|\ u_{\lambda_{0}} - u >\varepsilon \right\},
\qquad
F_{\varepsilon} = B_{\lambda_{0}} \setminus (\{0\}\cup E_\varepsilon).
\end{equation*}
For $\lambda > \lambda_{0}$, let $H_{\lambda} = B_{\lambda}\setminus B_{\lambda_{0}}$.
Then one has
\begin{equation*}
\Sigma_{\lambda}^{-} \subset
\bigl(\Sigma_{\lambda}^{-}\cap E_{\varepsilon}\bigr)
\cup
\bigl(\Sigma_{\lambda}^{-}\cap F_{\varepsilon}\bigr)
\cup
H_{\lambda}.
\end{equation*}
Clearly $\lim\limits_{\lambda \to \lambda_{0}} |H_{\lambda}| = 0$.
Since $u_{\lambda_{0}} > u$ in $B_{\lambda_{0}}\setminus\{0\}$,
it follows that $\lim\limits_{\varepsilon\to 0}|F_{\varepsilon}|=0$.
Thus, it remains to show that
\begin{equation*}
\lim_{\lambda \to \lambda_{0}}
|\Sigma_{\lambda}^{-}\cap E_{\varepsilon}| = 0.
\end{equation*}
Observe that for $x \in \Sigma_{\lambda}^{-}\cap E_{\varepsilon}$,
\begin{equation*}
u - u_{\lambda}
= (u - u_{\lambda_{0}}) + (u_{\lambda_{0}} - u_{\lambda}) > 0,
\end{equation*}
which implies
\begin{equation*}
u_{\lambda_{0}} - u_{\lambda} > u_{\lambda_{0}} - u \geq \varepsilon,
\qquad x \in \Sigma_{\lambda}^{-}\cap E_{\varepsilon}.
\end{equation*}
Hence,
\begin{equation*}
E_{\varepsilon}\cap \Sigma_{\lambda}^{-}
\subset
\left\{ x \in B_{\lambda_{0}}\setminus\{0\} \ \middle|\
u_{\lambda_{0}} - u_{\lambda} \geq \varepsilon \right\}
=: G_{\lambda}.
\end{equation*}
By Chebyshev's inequality,
\begin{equation*}
|G_{\lambda}| \leq
\frac{1}{\varepsilon^{\tau+1}}
\int_{B_{\lambda_{0}}\setminus\{0\}}
|u_{\lambda_{0}} - u_{\lambda}|^{\tau+1} .
\end{equation*}
For fixed $\varepsilon$, letting $\lambda \to \lambda_{0}$ gives
$|G_{\lambda}|\to 0$. Therefore, we can fix a small $\varepsilon>0$ and then choose $\delta>0$ sufficiently small such that
\begin{equation*}
C \left( \int_{\Sigma_{\lambda}^{-}}
u^{\tau+1} dy \right)^{\frac{2\sigma}{n}}
< \frac{1}{2}, \qquad
\lambda \in [\lambda_{0}, \lambda_{0}+\delta].
\end{equation*}
Consequently, $\Sigma_{\lambda}^{-}=\emptyset$ and
\begin{equation*}
u \leq u_\lambda \quad \text{for all } \lambda \in [\lambda_0, \lambda_0 + \delta].
\end{equation*}
This contradicts the definition of $\lambda_0$ as a supremum.
We conclude that
\begin{equation*}
\lambda_0 = +\infty.
\end{equation*}

\noindent{\it Step 3. $|x|^{\frac{n-2\sigma}{2}} u(x)$ is increasing along each ray.}

By Step 2, for all $\lambda > 0$ and $|x| \leq \lambda$, we have
\begin{equation*}
|x|^{\frac{n-2\sigma}{2}} u(x)
\leq \left( \frac{\lambda^2}{|x|} \right)^{\frac{n-2\sigma}{2}}
       u\!\left( \frac{\lambda^2 x}{|x|^2} \right).
\end{equation*}
Define
\begin{equation*}
W(x) := |x|^{\frac{n-2\sigma}{2}} u(x).
\end{equation*}
Then
\begin{equation*}
W(x) \leq W\!\left( \frac{\lambda^2 x}{|x|^2} \right),
\qquad \forall\, |x| \leq \lambda, \;\; \forall\, \lambda > 0.
\end{equation*}
Now fix $0 < t_1 < t_2$ and $\xi \in \mathbb{R}^n$ with $|\xi|=1$.
Take $\lambda$ and $|x|$ such that
\begin{equation*}
t_1 = |x|, \qquad t_2 = \frac{\lambda^2}{|x|}, \qquad \xi = \frac{x}{|x|}.
\end{equation*}
It follows that
\begin{equation*}
W(t_1 \xi) \leq W(t_2 \xi).
\end{equation*}
Thus $W(x)$ is increasing along each ray $\{\, t \xi : t \geq 0 \,\}$ where $|\xi| = 1$.
\end{proof}

\begin{lemma}\label{lemma infty}
    Assume that $n\ge 2$ and $0< \sigma< n/2$. Let $u$ and $K$ be given as in Theorem \ref{thm Lin}. If $K(\infty)=+\infty$, then \eqref{sobolev eq 7} admits no positive solution.
\end{lemma}

\begin{proof}
Suppose that $u$ is a positive solution. Denote $m(x)=\min\limits_{\partial B_{|x|}} K$. Then there exists $R_0>1$ such that
\begin{equation*}
    m(x)>1,\quad \forall\, |x|\ge R_0.
\end{equation*}
By Lemma \ref{lemma 3.1}, we have
\begin{equation*}
u(x)\ge (\min_{\partial B_{1}} u)\,|x|^{-\frac{n-2\sigma}{2}}
= C_{0}\,|x|^{-\frac{n-2\sigma}{2}},
\qquad \forall\,|x|\ge R_0.
\end{equation*}
Then for any $|x|\ge R_0$,
\begin{equation*}
\begin{aligned}
u(x)
&= C_{n,\sigma}\int_{\mathbb{R}^n}\frac{K(y)\,u^{\tau}(y)}{|x-y|^{\,n-2\sigma}}\,dy  \ge C_{n,\sigma}\,m(x)
\int_{|x|\le |y|\le 2|x|}
\frac{dy}{|y|^{\frac{n+2\sigma}{2}}\,|x-y|^{\,n-2\sigma}} \\
&\ge C_{n,\sigma}m(x)\, |x|^{-(n-2\sigma)}  \int_{|x|\le |y|\le 2|x|}
\frac{dy}{|y|^{\frac{n+2\sigma}{2}}}= C_{n,\sigma}\,m(x)\,|x|^{-\frac{n-2\sigma}{2}} .
\end{aligned}
\end{equation*}
Repeating this process, we have
\begin{equation*}
\begin{aligned}
u(x)
&= C_{n,\sigma}\int_{\mathbb{R}^n}\frac{K(y)\,u^{\tau}(y)}{|x-y|^{\,n-2\sigma}}\,dy \ge C_{n,\sigma}\,m(x)\int_{|x|\le |y|\le 2|x|}
\frac{m^{\tau}(y)}{|y|^{\frac{n+2\sigma}{2}}\,|x-y|^{\,n-2\sigma}}\,dy \\
&\ge C_{n,\sigma}\,m^{\tau+1}(x)\,|x|^{-\frac{n-2\sigma}{2}},\quad \text{for any $|x|\ge R_0$}.
\end{aligned}
\end{equation*}
After repeating this process $q$ times, we have
\begin{equation*}
\begin{aligned}
u(x)\ge C_{n,\sigma}\,(m(x))^\frac{\tau^q-1}{\tau-1}\,|x|^{-\frac{n-2\sigma}{2}},\quad \text{for any $|x|\ge R_0$}.
\end{aligned}
\end{equation*}
Thus
\begin{equation*}
u(0)
=  C_{n,\sigma} \int_{\mathbb{R}^n}\frac{K(y)\,u^{\tau}(y)}{|y|^{\,n-2\sigma}}\,dy
\ge C_{n,\sigma}\, (m(R_0))^{\frac{\tau(\tau^q-1)}{\tau-1}+1}\,\int_{|y|\ge R_0}
\frac{dy }{|y|^{\frac32n-\sigma}}\to \infty,
\end{equation*}
as $q\to\infty$, which contradicts $u\in C(\mathbb{R}^n)$.
\end{proof}

\begin{lemma}\label{lemma 3.2}
Assume that $n\ge 2$ and $0<\sigma<n/2$. Let $u$ and $K$ be given as in Theorem \ref{thm Lin}. If $u$ is a positive solution and $K(\infty)\in(0,+\infty)$, then
$|x|^{\frac{n-2\sigma}{2}} u(x)$ is uniformly bounded in $\mathbb{R}^n$.
\end{lemma}

\begin{proof}
Suppose this is not the case, i.e., there exists $x_i \to \infty$ such that
\begin{equation*}
|x_i|^{\frac{n-2\sigma}{2}} u(x_i) \to +\infty.
\end{equation*}
Define
\begin{equation*}
v(y) := \frac{1}{|y|^{n-2\sigma}} u\!\left( \frac{y}{|y|^2} \right), \quad y \neq 0,
\end{equation*}
which satisfies
\begin{equation*}
v(x)   = \int_{\mathbb{R}^n} \frac{K^*(y)}{|x-y|^{n-2\sigma}} v^\tau(y)\,dy, \quad x \neq 0,
\end{equation*}
where $K^*(y) = K\!\left( \frac{y}{|y|^2} \right)$.
Let $y_i = \frac{x_i}{|x_i|^2}$. Then
\begin{equation*}
    |y_i|^{\frac{n-2\sigma}{2}} v(y_i) \;\;\xrightarrow[i\to\infty]{}\; +\infty.
\end{equation*}
Let $l_i = \frac{1}{2}|y_i|$ and set
\begin{equation*}
S(y) := v(y)\,(l_i - |y-y_i|)^{\frac{n-2\sigma}{2}}, \qquad y \in B(y_i,l_i).
\end{equation*}
Choose $\bar{y}_i \in B(y_i,l_i)$ such that
\begin{equation*}
S(\bar{y}_i) = \max_{B(y_i,l_i)} S(y) > S(y_i)
= \Big(\frac{1}{2}|y_i|\Big)^{\frac{n-2\sigma}{2}} v(y_i) \;\;\xrightarrow[i\to\infty]{}\; +\infty.
\end{equation*}
Let $M_i = v(\bar{y}_i)$ and define the rescaling
\begin{equation*}
\tilde{v}_i(x) := M_i^{-1} v\!\bigg( \bar{y}_i + M_i^{-\frac{2}{n-2\sigma}} x \bigg).
\end{equation*}
Then
\begin{equation*}
\tilde{v}_i(x) = \int_{\mathbb{R}^n}
   \frac{K_i(y)\,\tilde{v}_i^\tau(y)}{|x-y|^{\,n-2\sigma}}\,dy,
   \qquad |x| \leq M_i^{\frac{2}{n-2\sigma}} l_i,
\end{equation*}
with
\begin{equation*}
K_i(y) = K^*\!\bigg( \bar{y}_i + M_i^{-\frac{2}{n-2\sigma}} y \bigg).
\end{equation*}
For any $R>0$ and $|x|\leq R$, set $z = \bar{y}_i + M_i^{-\frac{2}{n-2\sigma}}x$. Then
\begin{equation*}
\tilde{v}_i(x) = \frac{v(z)}{v(\bar{y}_i)}
= \frac{S(z)}{S(\bar{y}_i)}
  \left(\frac{l_i-|\bar{y}_i-y_i|}{l_i-|z-y_i|}\right)^{\frac{n-2\sigma}{2}}
\leq 2^{\frac{n-2\sigma}{2}},
\end{equation*}
since $0 < \frac{S(z)}{S(\bar{y}_i)} \leq 1$ and
\begin{equation*}
\begin{aligned}
  l_i - |z-y_i| &\geq  l_i-|\bar{y}_i-y_i|-|\bar{y}_i-z|= l_i-|\bar{y}_i-y_i|-M_i^{\frac{2}{n-2\sigma}}|x|\\
 & \geq\frac{1}{2}\big(l_i - |\bar{y}_i-y_i|\big), \qquad \text{for large } i.
\end{aligned}
\end{equation*}
By Jin-Li-Xiong \cite[proof of Proposition~2.9]{jin-nirenberg-2017}, after passing to a subsequence,
\begin{equation*}
\tilde{v}_i \to U_0 \quad \text{in } C^\alpha_{\mathrm{loc}}(\mathbb{R}^n), \qquad \text{for some }U_0 \in C^\alpha(\mathbb{R}^n) \text{ and }0<\alpha<1,
\end{equation*}
and $U_0$ satisfies
\begin{equation*}
U_0(x) = \int_{\mathbb{R}^n} \frac{K(\infty)\,U_0^\tau(y)}{|x-y|^{\,n-2\sigma}}\,dy,
\end{equation*}
where $K(\infty)=\lim\limits_{|x|\to\infty} K(x)$.
By Chen-Li-Ou \cite{CLO3} or Li \cite{Ly}, the solution has the form
\begin{equation*}
U_0(x) = c\left(\frac{t}{t^2+|x-x_0|^2}\right)^{\frac{n-2\sigma}{2}}.
\end{equation*}
Thus $U_0$ has a maximum at $x_0$, so $\tilde{v}_i$ attains a local maximum at some $\eta_i\to x_0$.
Returning to $v$, this implies a local maximum at
\begin{equation*}
y_i^* = \bar{y}_i + M_i^{-\frac{2}{n-2\sigma}}\eta_i.
\end{equation*}
Since
\begin{equation*}
\frac{\partial}{\partial t}\big(v(ty_i^*)\,t^{\frac{n-2\sigma}{2}}\big)\leq 0 ~~~~ \text{and} ~~~~
\frac{\partial}{\partial t}\big(v(ty_i^*)\big)\Big|_{t=1}=0,
\end{equation*}
Leibniz' rule yields
\begin{equation*}
v(y_i^*) \leq 0,
\end{equation*}
a contradiction.
\end{proof}

\begin{lemma}\label{lemma 3.3}
Assume that $n\ge 2$ and $0< \sigma <n/2$. Let $u$ and $K$ be given as in Theorem \ref{thm Lin}. If $u$ is a positive solution and $K(\infty)\in(0,+\infty)$, then
$\lim\limits_{|x|\to\infty} |x|^{\frac{n-2\sigma}{2}} u(x)$ exists.
\end{lemma}

\begin{proof}
Let $x_i \to \infty$ be any sequence and define
\begin{equation*}
u_i(y) := |x_i|^{\frac{n-2\sigma}{2}} u(|x_i|y).
\end{equation*}
By Lemma \ref{lemma 3.2},
\begin{equation*}
u_i(y) \leq C |y|^{-\frac{n-2\sigma}{2}} \qquad \text{for some constant $C$}.
\end{equation*}
Using the same process as in Lemma \ref{lemma 3.2}, $u_i$ converges uniformly (up to a subsequence)
to a function $U(y)$ in any compact subset of $\mathbb{R}^n\setminus\{0\}$, where $U$ is a positive solution of
\begin{equation*}
U(x) = \int_{\mathbb{R}^n} \frac{K(\infty)\, U^\tau(y)}{|x-y|^{\,n-2\sigma}}\,dy,
\qquad x \in \mathbb{R}^n\setminus\{0\}.
\end{equation*}
By Chen-Li-Ou \cite[Theorem~2]{chen2003qualitative}, $U$ is radially symmetric about the origin. Define
\begin{equation*}
V(t) := r^{\frac{n-2\sigma}{2}} U(r), \qquad r=|x|,\; t=\ln r,\; \theta=\frac{x}{|x|}.
\end{equation*}
Then $V$ satisfies (see Jin-Xiong \cite{jin_asymptotic_2021})
\begin{equation*}
V(t) = K(\infty) \int_{-\infty}^{\infty} J(s) V^\tau(t-s)\,ds,
\end{equation*}
with kernel
\begin{equation*}
J(t) := \int_{\mathbb{S}^{n-1}} \frac{d\xi}{(e^t+e^{-t}-2\theta\cdot\xi)^{\frac{n-2\sigma}{2}}},
\end{equation*}
and
\begin{equation*}
    \int_{-\infty}^{\infty} J(t)\,dt = C(n,\sigma)<\infty.
\end{equation*}
By Lemmas \ref{lemma 3.1} and \ref{lemma 3.2}, $V(t)$ is bounded and nondecreasing, so the limits
\begin{equation*}
A_1 := \lim_{t\to -\infty} V(t) \quad  \textmd{ and } \quad  A_2 := \lim_{t\to +\infty} V(t)
\end{equation*}
exist. By the dominated convergence theorem,
\begin{equation*}
A_1 = K(\infty) C(n,\sigma) A_1^\tau,
\qquad
A_2 = K(\infty) C(n,\sigma) A_2^\tau.
\end{equation*}
Since $A_1, A_2 > 0$, we obtain
\begin{equation*}
A_1 = A_2 = \big( K(\infty) C(n,\sigma) \big)^{-\frac{1}{\tau-1}}.
\end{equation*}
Consequently,
\begin{equation*}
\lim_{|x|\to\infty} |x|^{\frac{n-2\sigma}{2}} u(x)
= \big( K(\infty) C(n,\sigma) \big)^{-\frac{n-2\sigma}{4\sigma}}.
\end{equation*}
Lemma \ref{lemma 3.3} is established.
\end{proof}

\begin{lemma}\label{lemma 3.4}
Let $u$ and $K$ be given as in Theorem \ref{thm Lin}. If $K(\infty)\in(0,+\infty)$ and $\sigma = m \in (0, n/2)$ is a positive integer, then \eqref{sobolev eq 7} admits no positive solution.
\end{lemma}

\begin{proof}
Suppose $u$ is a positive solution. By Guo-Peng-Yan  \cite[(3.2) and (3.9)]{guo2015existencelocaluniquenessbubbling}, the Pohozaev identity reads
\begin{equation}\label{eq 4.9}
    2\int_{B_r}(x\cdot \nabla u)(-\Delta)^m u = \int_{\partial B_r} g_m(u,u) - (n-2m) \int_{B_r} u(-\Delta)^m u.
\end{equation}
 $g_m(u,v)$ has the following form:
\begin{equation*}
g_m(u,v)
=
\sum_{j=1}^{2m-1} l_j\!\left(y-x,\nabla^{j}u,\nabla^{2m-j}v\right)
\;+\;
\sum_{j=0}^{2m-1} \tilde{l}_j\!\left(\nabla^{j}u,\nabla^{2m-j-1}v\right),
\end{equation*}
where $l_j\!\left(y-x,\nabla^{j}u,\nabla^{2m-j}v\right)$ and
$\tilde{l}_j\!\left(\nabla^{j}u,\nabla^{2m-j-1}v\right)$
are linear in each component.

Since $u$ solves $(-\Delta)^m u(x) = K(x) u^{\frac{n+2m}{n-2m}}(x)$, an integration by parts yields
\begin{equation}\label{sobolev eq 9}
    \frac{n-2m}{n} \int_{B_r}(x\cdot \nabla K) u^{\frac{2n}{n-2m}}
    = \frac{n-2m}{n} r \int_{\partial B_r} K u^{\frac{2n}{n-2m}} - \int_{\partial B_r} g_m(u,u)
\end{equation}
for any $r>0$. Denote the right-hand side as $P(m,r,u)$:
\begin{equation}\label{P}
    P(m,r,u) = \frac{n-2m}{n} r \int_{\partial B_r} K u^{\frac{2n}{n-2m}} - \int_{\partial B_r} g_m(u,u).
\end{equation}
The monotonicity of $K$ ensures that the left-hand side of \eqref{sobolev eq 9} is positive for any $r>0$. Hence, we only need to verify that
\begin{equation*}\label{eq 4.6}
    P(m,r,u) \le 0 \quad \text{for $r$ large},
\end{equation*}
and then reach a contradiction.
By Lemma \ref{lemma 3.3}, we have
\begin{equation*}
    u(x)=M_0|x|^{-\frac{n-2m}{2}}(1+o(1))\qquad\text{for $|x|$ large.}
\end{equation*}
By Jin-Li-Xiong \cite[Theorems 2.4 and 2.5]{jin-nirenberg-2017} and standard regularity arguments, one derives that
\begin{equation*}
    D^\beta u(x)=M_0D^\beta (|x|^{-\frac{n-2m}{2}})(1+o(1))\qquad\text{for $|x|$ large,}
\end{equation*}
where $\beta$ is the multi-index. Denote $u_0(x)=M_0|x|^{-\frac{n-2m}{2}}$. Then
\begin{equation}\label{eq 3.4}
    P(m,r,u)= P(m,r,u_0) +o(1)\qquad \text{as }r\to \infty.
\end{equation}
We only need to verify that
\begin{equation*}
     P(m,r,u_0)<0\qquad \text{for any }r>0.
\end{equation*}
For the first term in $P(m,r,u_0)$, one has
\begin{equation}\label{P1}
    \frac{n-2m}{n} r \int_{\partial B_r} (M_0|x|^{-\frac{n-2m}{2}})^{\frac{2n}{n-2m}}K(x)\, dx
    \leq \frac{n-2m}{n} M_0^{\frac{2n}{n-2m}} |\mathbb{S}^{n-1}| K(\infty).
\end{equation}
By the calculation in Appendix \ref{sec gm}, we have
\begin{equation}\label{gm}
    \int_{\partial B_r} g_m(u_0,u_0)
    = |\mathbb{S}^{n-1}| M_0^2 \prod_{i=0}^{m-1}\Big(\frac{n-2m+4i}{2}\Big)^2.
\end{equation}
Since $u_0$ solves $(-\Delta)^m u_0 = K(\infty) u_0^{\frac{n+2m}{n-2m}}$ with $K(\infty) = \lim\limits_{|x|\to\infty} K(x)$, we have
\begin{equation}\label{m0}
    M_0^{\frac{4m}{n-2m}} = {\prod_{i=0}^{m-1}\big(\frac{n-2m+4i}{2}\big)^2 }/{K(\infty)}.
\end{equation}
Combining \eqref{P}-\eqref{m0}, we obtain
\begin{equation*}
\begin{aligned}
    P(m,r,u) &=P(m,r,u_0)+o(1)\\
   & \leq M_0^2 |\mathbb{S}^{n-1}| \Big( \frac{n-2m}{n} M_0^{\frac{4m}{n-2m}}  K(\infty) - \prod_{i=0}^{m-1}\big(\frac{n-2m+4i}{2}\big)^2 \Big)+o(1) \\
    &= M_0^2 |\mathbb{S}^{n-1}| \prod_{i=0}^{m-1}\Big(\frac{n-2m+4i}{2}\Big)^2 \Big( \frac{n-2m}{n} - 1 \Big) +o(1)\\
    &\le 0, \quad \text{for $r$ sufficiently large}.
\end{aligned}
\end{equation*}
Thus we reach a contradiction.
\end{proof}

\begin{lemma}\label{lemma 3.5}
Let $u$ and $K$ be given as in Theorem \ref{thm Lin}. If $K(\infty)\in(0,+\infty)$ and $\sigma \in (0,1)$, then \eqref{sobolev eq 7} admits no positive solution.
\end{lemma}

\begin{proof} Suppose $u$ is a positive solution. Let
\begin{equation*}
\mathcal{P}_\sigma(x,t) = \beta(n,\sigma) \frac{t^{2\sigma}}{(|x|^2+t^2)^{\frac{n+2\sigma}{2}}}, \quad (x,t) \in \mathbb{R}_+^{n+1},
\end{equation*}
where $\beta(n,\sigma)$ is chosen so that
\begin{equation*}
\beta(n,\sigma) \int_{\mathbb{R}^n} (|x|^2+1)^{-\frac{n+2\sigma}{2}} dx = 1.
\end{equation*}
Define
\begin{equation*}
U(x,t) = (\mathcal{P}_\sigma * u)(x,t) = \beta(n,\sigma) \int_{\mathbb{R}^n} \frac{t^{2\sigma}}{(|x-y|^2+t^2)^{\frac{n+2\sigma}{2}}} u(y) dy.
\end{equation*}
Let $b=1-2\sigma$, and
\begin{equation*}
\Delta_b := \Delta_{x,t} + \frac{b}{t} \partial_t = t^{-b} \operatorname{div}(t^b \nabla).
\end{equation*}
Then it is well-known that
\begin{equation}\label{extension 1}
\begin{aligned}
U(x,0) &=u(x),\\
\Delta_b U &= 0 \quad\text{ in } \mathbb{R}_+^{n+1}, \\
- \lim_{t\to 0} t^b \partial_t  U &= N_{n,\sigma} (-\Delta)^\sigma u(x),
\end{aligned}
\end{equation}
where $N_{n,\sigma}>0$ depends only on $n,\sigma$. The Pohozaev identity (e.g., Jin-Xiong \cite[Proposition  A.3]{jin_asymptotic_2021}) reads
\begin{equation}\label{pohozaev 2}
\begin{aligned}
\int_E \operatorname{div}(t^b \nabla  U) \langle X, \nabla U \rangle
&= -\frac{n-2\sigma}{2} \int_E \operatorname{div}(t^b \nabla  U) U + \frac{n-2\sigma}{2} \int_{\partial E} t^b \partial_\nu  U U\\
&+ \int_{\partial E} t^b \partial_\nu  U \langle X,\nabla U\rangle
+ \int_{\partial E} Q_0(U),
\end{aligned}
\end{equation}
where $\nu$ is the outward unit normal of $E$ and
\begin{equation}\label{q 1}
\int_{\partial E} Q_0(U) := -\frac12 \int_{\partial E} t^b |\nabla U|^2 X^k \nu_k.
\end{equation}
Let $\mathcal{B}_R$ be the ball of radius $R$ centered at $0$ in $\mathbb{R}^{n+1}$, $\mathcal{B}_R^+ = \mathcal{B}_R \cap \mathbb{R}_+^{n+1}$, $\partial' \mathcal{B}_R$ the flat part (i.e. $B_R$ in $\mathbb{R}^n$) of $\partial\mathcal{B}_R^+ $, and $\partial'' \mathcal{B}_R = \partial \mathcal{B}_R \cap \{t>0\}$. Since $\operatorname{div}(t^b \nabla  U) = t^b \Delta_bU=0$, we have
\begin{equation*}
\int_E \operatorname{div}(t^b \nabla  U) \langle X, \nabla U\rangle = 0 \quad \text{and }-\frac{n-2\sigma}{2} \int_E \operatorname{div}(t^b \nabla  U) U = 0.
\end{equation*}
Using \eqref{extension 1} and \eqref{q 1}, one verifies
\begin{equation*}
\int_{\partial' \mathcal{B}_R} Q_0(U) = 0.
\end{equation*}
Thus \eqref{pohozaev 2} reduces to
\begin{equation}\label{eq 2.3}
\frac{n-2\sigma}{2} \int_{\partial'\mathcal{B}_R \cup \partial''\mathcal{B}_R} t^b U\partial_\nu  U
+ \int_{\partial'\mathcal{B}_R \cup \partial''\mathcal{B}_R} t^b \partial_\nu  U \langle X, \nabla U \rangle
+ \int_{\partial'' \mathcal{B}_R} Q_0(U) = 0.
\end{equation}
Assume $u$ solves $(-\Delta)^\sigma u(x) = K(x) u^{\frac{n+2\sigma}{n-2\sigma}}(x)$. Note that $\partial_\nu U=-\partial_t U$ when $X\in \partial' \mathcal{B}_R$. By \eqref{extension 1},
\begin{equation}\label{eq 2.1}
\frac{n-2\sigma}{2} \int_{\partial' \mathcal{B}_R} t^b U\partial_\nu  U  = N_{n,\sigma} \frac{n-2\sigma}{2} \int_{B_R} K u^{\frac{2n}{n-2\sigma}}.
\end{equation}
Integration by parts yields
\begin{equation}\label{eq 2.2}
\begin{aligned}
\int_{\partial' \mathcal{B}_R} t^b \partial_\nu  U \langle X, \nabla U\rangle &= N_{n,\sigma}\Bigg( \frac{n-2\sigma}{2n} R \int_{\partial B_R} K u^{\frac{2n}{n-2\sigma}}\\
&- \frac{n-2\sigma}{2n}  \int_{B_R} u^{\frac{2n}{n-2\sigma}} (\nabla K \cdot x)
- \frac{n-2\sigma}{2} \int_{B_R} K u^{\frac{2n}{n-2\sigma}} \Bigg).
\end{aligned}
\end{equation}
Combining \eqref{eq 2.3}-\eqref{eq 2.2}, we get
\begin{equation}\label{eq 2.4}
\begin{aligned}
& N_{n,\sigma} \frac{n-2\sigma}{2n} \int_{B_R} u^{\frac{2n}{n-2\sigma}} (\nabla K \cdot x) = N_{n,\sigma} \frac{n-2\sigma}{2n} R \int_{\partial B_R} K u^{\frac{2n}{n-2\sigma}} \\
& \qquad +\left[ \frac{n-2\sigma}{2} \int_{\partial'' \mathcal{B}_R} t^b U \partial_\nu  U  \right. + \left.\int_{\partial'' \mathcal{B}_R} t^b \partial_\nu  U \langle X, \nabla U \rangle
+ \int_{\partial'' \mathcal{B}_R} Q_0(U)\right]
\end{aligned}
\end{equation}
for any $R>0$. Denote the right-hand side of \eqref{eq 2.4} as $P(\sigma,R,u)$.
The monotonicity of $K$ ensures that the left-hand side of \eqref{eq 2.4} is positive for any $R>0$. Hence, we only need to verify that
\begin{equation*}
    P(\sigma,R,u) \le 0 \quad \text{for $R$ large},
\end{equation*}
and then reach a contradiction.

By Lemma \ref{lemma 3.3}, we have
\begin{equation*}
    u(x) = M_0 |x|^{-\frac{n-2\sigma}{2}}(1+o(1)) \qquad \text{as } |x| \to \infty.
\end{equation*}
Denote $u_0(x)=M_0|x|^{-\frac{n-2\sigma}{2}}$ and $U_0(x,t) = (\mathcal{P}_\sigma * u_0)(x,t) $. Then
\begin{equation}\label{eq 3.5}
   P(\sigma,R,u)= P(\sigma,R,u_0) +o(1)\qquad \text{as }R\to \infty.
\end{equation}
We only need to verify that
\begin{equation*}
    P(\sigma,R,u_0)<0\qquad \text{for any }R>0.
\end{equation*}
By Lemma~\ref{appendix U},
\begin{equation*}
U_0(x,t) = \frac{\Gamma^2\!\left(\frac{n+2\sigma}{4}\right)}{\Gamma\!\left(\frac{n}{2}\right)\Gamma(\sigma)} M_0 (|x|^2+t^2)^{-\frac{n-2\sigma}{4}} {}_2F_1\!\left(\frac{n-2\sigma}{4},\frac{n-2\sigma}{4};\frac{n}{2};\frac{|x|^2}{|x|^2+t^2}\right).
\end{equation*}
Let
\begin{equation*}
r^2 = |x|^2 + t^2, \quad z = \frac{|x|^2}{r^2}, \quad F(z) = {}_2F_1\!\left(\frac{n-2\sigma}{4}, \frac{n-2\sigma}{4}; \frac{n}{2}; z\right)
\end{equation*}
and
\begin{equation}\label{eq C0}
    C_0 = \frac{\Gamma^2\!\left(\frac{n+2\sigma}{4}\right)}{\Gamma\!\left(\frac{n}{2}\right)\Gamma(\sigma)}.
\end{equation}
Then
\begin{equation*}
U_0(x,t) = C_0 M_0 r^{-\frac{n-2\sigma}{2}} F(z).
\end{equation*}
By Lemma \ref{reduction}, $P(\sigma,R,u_0)$ reduces to
\begin{equation}\label{eq 2.5}
\begin{aligned}
P(\sigma,R,u_0)
&=  N_{n,\sigma} \frac{n-2\sigma}{2n} R \int_{\partial B_R} K u_0^{\frac{2n}{n-2\sigma}}
+ \int_{\partial'' \mathcal{B}_R} Q_0(U_0).
\end{aligned}
\end{equation}
By \eqref{eq Q0} and \eqref{eq N} in Appendix \ref{CC}, we have
\begin{equation*}
\begin{aligned}
\setlength{\jot}{12pt}
P(\sigma,R,u) &= P(\sigma,R,u_0)+o(1)\\
&\leq \frac{n-2\sigma}{n} |\mathbb{S}^{n-1}| M_0^2 C_0 \frac{\Gamma(\frac{n}{2}) \Gamma(1-\sigma)}{\Gamma^2(\frac{n-2\sigma}{4})} -C_0^2 M_0^2 |\mathbb{S}^{n-1}| \frac{\Gamma^2(\frac{n}{2}) \Gamma(\sigma) \Gamma(1-\sigma)}{\Gamma^2(\frac{n+2\sigma}{4}) \Gamma^2(\frac{n-2\sigma}{4})}+o(1),\\
&= M_0^2 C_0 |\mathbb{S}^{n-1}| \frac{\Gamma(1-\sigma) \Gamma(\frac{n}{2})}{\Gamma^2(\frac{n-2\sigma}{4})}\left(\frac{n-2\sigma}{n} - 1\right)+o(1) \le 0\qquad \text{for $R$ large}.
\end{aligned}
\end{equation*}
The equality follows from \eqref{eq C0}.
Thus a contradiction is reached.
\end{proof}

Theorem \ref{thm Lin} follows directly from Lemmas \ref{lemma infty}, \ref{lemma 3.4} and \ref{lemma 3.5}.
\end{proof}

\section{\texorpdfstring{Liouville theorem for sign-changing symmetric $Q$-curvature}{Liouville theorem for sign-changing symmetric $Q$-curvature}}\label{sign-changing}

In this section, we prove Theorem \ref{thm:no-asymptotic} by the method of moving spheres in integral form. Throughout this section, set $\tau:=\frac{n+2\sigma}{n-2\sigma}$. Based on the integral representation, we first establish a decay estimate for the solution at infinity.

\begin{lemma}\label{lem:no-asymptotic-decay}
Under the assumptions of Theorem \ref{thm:no-asymptotic}, $u$ satisfies
\begin{equation}\label{eq:no-asymptotic-integral}
    u(x)=C_{n,\sigma}\int_{\mathbb R^n}
    \frac{K(y)u^\tau(y)}{|x-y|^{n-2\sigma}}\,dy,
    \quad x\in\mathbb R^n.
\end{equation}
Moreover, there exists $C>0$ such that
\begin{equation}\label{eq:natural-upper-decay}
    u(x)\le C|x|^{-(n-2\sigma)},
    \qquad |x|\ge 2.
\end{equation}
\end{lemma}

\begin{proof}
Let $f(x, t) = K(x) t^\tau$. Then it is easy to verify that all the assumptions of Theorem \ref{thm equivalence} are fulfilled, and consequently \eqref{eq:no-asymptotic-integral} follows. Since $K(y) \leq 0$ for $|y|>1$, it follows that, for $|x|\ge2$,
\[
\begin{aligned}
    u(x)
    & \le C_{n,\sigma}\int_{B_1} \frac{K(y)u^\tau(y)}{|x-y|^{n-2\sigma}}\,dy \\
    & \le C|x|^{-(n-2\sigma)}\int_{B_1}K(y)u^\tau(y)\,dy \\
    & \le C|x|^{-(n-2\sigma)}.
\end{aligned}
\]
This proves \eqref{eq:natural-upper-decay}.
\end{proof}

\begin{proof}[Proof of Theorem \ref{thm:no-asymptotic}]
Suppose, for contradiction, that $u$ is a nontrivial nonnegative solution. For $\lambda>0$, let
\[
    x^\lambda:=\frac{\lambda^2x}{|x|^2}
    \quad  \mbox{ and }  \quad
    u_\lambda(x):=
    \left(\frac{\lambda}{|x|}\right)^{n-2\sigma} u(x^\lambda), ~~ x\ne0.
\]
If $u\equiv0$ in $B_1$, then \eqref{eq:no-asymptotic-integral} together with $K\le0$ in $B_1^c$ would give
\[
u(x)
 =C_{n,\sigma}\int_{B_1^c}
 \frac{K(y)u^\tau(y)}{|x-y|^{n-2\sigma}}\,dy
 \le0,
 \qquad x\in\mathbb R^n.
\]
Since $u\ge0$, it follows that $u\equiv0$ in $\mathbb R^n$, a contradiction to the nontriviality of $u$. Therefore, there exists a point $y_0\in B_1$ and a neighborhood $E$ of $y_0$ such that $u>0$ in $E$. Moreover, one can verify that
\begin{equation}\label{eq 5.3}
\begin{aligned}
&u(x) - u_{\lambda}(x)
= \\ &\int_{B_{\lambda}}
\bigg( \frac{1}{|x-y|^{n-2\sigma}} - \Big(\frac{\lambda}{|x|}\Big)^{n-2\sigma} \frac{1}{\left|x^\lambda - y\right|^{n-2\sigma}} \bigg)
\left( K(y)u^{\tau}(y) - K\!(y^\lambda) u^{\tau}_\lambda(y) \right) dy,
\end{aligned}
\end{equation}
where
\begin{equation}\label{eq:section6-kernel-positive}
    \frac{1}{|x-y|^{n-2\sigma}} - \Big(\frac{\lambda}{|x|}\Big)^{n-2\sigma} \frac{1}{\left|x^\lambda - y\right|^{n-2\sigma}} >0 \quad \text{for }x, y\in B_\lambda\backslash\{0\}.
\end{equation}
For $0<\lambda\le1$ and $y\in B_\lambda\setminus\{0\}$, we have $|y^\lambda|>|y|$. If $|y^\lambda|<1$, the radial monotonicity of $K$ gives $K(y)\ge K(y^\lambda)$. If $|y^\lambda|\ge1$, then $K(y)>0\ge K(y^\lambda)$. Hence
\begin{equation}\label{eq:section6-K-comparison}
    K(y)\ge K(y^\lambda),
    \quad y\in B_\lambda\setminus\{0\},~~ 0<\lambda\le1.
\end{equation}
For $\lambda=1$, we have
\[
K(y)u^\tau(y)-K\Big(\frac{y}{|y|^2} \Big)u_1^\tau(y)
\ge K(y)u^\tau(y)\ge0
\quad\text{for }y\in B_1\setminus\{0\}.
\]
Moreover, this quantity is strictly positive for every $y\in E$,
because $K>0$ in $B_1$ and $u>0$ in $E$. Therefore,
\eqref{eq 5.3} and \eqref{eq:section6-kernel-positive} imply that \[
u(x)-u_1(x)>0,
\quad x\in B_1\setminus\{0\}.
\]
In particular, we conclude that
\begin{equation}\label{eq:section6-interior-positivity}
u>0\quad\text{in }B_1\setminus\{0\}.
\end{equation}

Define
\[
\begin{aligned}
    \lambda_0:=\inf\bigl\{\lambda\in(0,1]:
    u\ge u_\rho\text{ in }B_\rho\setminus\{0\}
    \text{ for every }\rho\in[\lambda,1]\bigr\}.
\end{aligned}
\]
We claim that
$\lambda_0=0$. Otherwise, $\lambda_0>0$. First we have
\begin{equation}\label{eq:section6-critical-comparison}
    u\ge u_{\lambda_0}
    \quad\text{in }B_{\lambda_0}\setminus\{0\}.
\end{equation}
The inequality is in fact strict. Indeed, by
\eqref{eq:section6-K-comparison} and \eqref{eq:section6-critical-comparison}, we obtain
\[
    K(y)u^\tau(y)-K(y^{\lambda_0})u_{\lambda_0}^\tau(y)\ge0
    \qquad\text{in }B_{\lambda_0}\setminus\{0\}.
\]
Moreover, for $0<|y|<\lambda_0^2$, we have $u(y)>0$ and $|y^{\lambda_0}|>1$. Hence
\[
K(y)>0\ge K(y^{\lambda_0}),
\]
and so
\[
K(y)u^\tau(y)
 -K(y^{\lambda_0})u_{\lambda_0}^\tau(y)>0.
\]
Using \eqref{eq 5.3} and
\eqref{eq:section6-kernel-positive}, we conclude that
\begin{equation}\label{eq:section6-critical-strict}
    u>u_{\lambda_0}
    \quad\text{in }B_{\lambda_0}\setminus\{0\}.
\end{equation}
Once the strict inequality \eqref{eq:section6-critical-strict} has been established, the comparison can be extended slightly inside \(B_{\lambda_{0}}\).
More precisely, there exists a small \(\delta\in(0,\lambda_0/2)\) such that
\begin{equation}\label{eq:section6-continuation}
u(x) \geq u_{\lambda}(x), \quad x \in B_{\lambda}\setminus\{0\}, \ \lambda \in [\lambda_{0}-\delta,\lambda_{0}].
\end{equation}
To see this, for $0<\lambda\le1$, let
\[
    \Sigma_\lambda^-:=
    \bigl\{x\in B_\lambda\setminus\{0\}:u_\lambda(x)>u(x)\bigr\}.
\]
For $y\notin\Sigma_\lambda^-$, \eqref{eq:section6-K-comparison} and $u_\lambda(y)\le u(y)$ give
\[
    K(y^\lambda)u_\lambda^\tau(y)-K(y)u^\tau(y)\le0.
\]
For $y\in\Sigma_\lambda^-$, \eqref{eq:section6-K-comparison}  and the mean value theorem give
\[
\begin{aligned}
    K(y^\lambda)u_\lambda^\tau(y)-K(y)u^\tau(y)
    \le K(y)\bigl(u_\lambda^\tau(y)-u^\tau(y)\bigr)\le \tau K(y)u_\lambda^{\tau-1}(y)
       \bigl(u_\lambda(y)-u(y)\bigr).
\end{aligned}
\]
Consequently, for $x\in\Sigma_\lambda^-$, we have
\begin{align*}
0&<u_{\lambda}(x)-u(x) \\
 &\leq
\tau \int_{\Sigma_{\lambda}^-}
\!\bigg(   \frac{1}{|x-y|^{n-2\sigma}} - \Big(\frac{\lambda}{|x|}\Big)^{n-2\sigma} \frac{1}{\left|x^\lambda - y\right|^{n-2\sigma}} \bigg)
K(y) u_{\lambda}^{\tau-1}(y)\,\bigl(u_{\lambda}(y)-u(y)\bigr)\, dy \\[4pt]
&\leq C \int_{\Sigma_{\lambda}^-}\frac{u_{\lambda}^{\tau-1}(y)\,(u_{\lambda}(y)-u(y))}{|x-y|^{n-2\sigma}} dy,
\end{align*}
where $C=C(n,\sigma, \|K\|_{L^\infty(B_1)})$. By the Hardy-Littlewood-Sobolev inequality and H\"older's inequality, for every
$q>n/(n-2\sigma)$,
\begin{equation}\label{eq:section6-HLS}
\begin{aligned}
    \|u_\lambda-u\|_{L^q(\Sigma_\lambda^-)}
    \le C
    \left(\int_{\Sigma_\lambda^-}u_\lambda^{\tau+1}\,dy\right)^{2\sigma/n}
    \|u_\lambda-u\|_{L^q(\Sigma_\lambda^-)}.
\end{aligned}
\end{equation}
Note that \eqref{eq:natural-upper-decay} implies \begin{equation}\label{eq:section6-uniform-kelvin-bound}
    \sup_{\lambda\in[\lambda_0/2,\lambda_0]}
    \|u_\lambda\|_{L^\infty(B_{\lambda_0})}<\infty.
\end{equation}
Indeed, if $|x^\lambda|\ge 2$, then
\[
    u_\lambda(x)=\left(\frac{\lambda}{|x|}\right)^{n-2\sigma}u(x^\lambda)
    \le C\left(\frac{\lambda}{|x|}\right)^{n-2\sigma}|x^\lambda|^{-(n-2\sigma)}
    =C\lambda^{-(n-2\sigma)}.
\]
If $|x^\lambda|<2$, then
\[
    u_\lambda(x)
    =\left(\frac{|x^\lambda|}{\lambda}\right)^{n-2\sigma} u(x^\lambda)
    \le \left(\frac{2}{\lambda}\right)^{n-2\sigma}
       \|u\|_{L^\infty(B_2)}.
\]
This proves \eqref{eq:section6-uniform-kelvin-bound}. It follows from the dominated convergence theorem that
\begin{equation}\label{eq:section6-Lp-convergence}
    u_\lambda\to u_{\lambda_0}
    \quad\text{in }L^{\tau+1}(B_{\lambda_0})
    \quad\text{as }\lambda\uparrow{\lambda_0}.
\end{equation}
For $\varepsilon>0$, set
\[
\begin{aligned}
    E_\varepsilon
    &:=\bigl\{x\in B_{\lambda_0}\setminus\{0\}:
          u(x)-u_{\lambda_0}(x)>\varepsilon\bigr\},\\
    F_\varepsilon
    &:=\bigl\{x\in B_{\lambda_0}\setminus\{0\}:
          0<u(x)-u_{\lambda_0}(x)\le\varepsilon\bigr\}.
\end{aligned}
\]
By \eqref{eq:section6-critical-strict} and the continuity of $u$,
\begin{equation}\label{eq:section6-F-small}
    |F_\varepsilon|\to 0
    \quad\text{as }\varepsilon\downarrow0.
\end{equation}
For $\lambda<{\lambda_0}$ and $x\in\Sigma_\lambda^-\cap E_\varepsilon$, one has
\[
    u_\lambda(x)-u_{\lambda_0}(x)
    =\bigl(u_\lambda(x)-u(x)\bigr)
     +\bigl(u(x)-u_{\lambda_0}(x)\bigr)>\varepsilon.
\]
Thus, by Chebyshev's inequality and \eqref{eq:section6-Lp-convergence},
\begin{equation}\label{eq:section6-E-small}
\begin{aligned}
    |\Sigma_\lambda^-\cap E_\varepsilon|
    &\le \varepsilon^{-(\tau+1)}
       \int_{B_{\lambda_0}}|u_\lambda-u_{\lambda_0}|^{\tau+1}\,dx
      \to 0
      \qquad\text{as }\lambda\uparrow{\lambda_0}.
\end{aligned}
\end{equation}
Since $\Sigma_\lambda^-\subset
(\Sigma_\lambda^-\cap E_\varepsilon)\cup F_\varepsilon$, estimates
\eqref{eq:section6-F-small} and \eqref{eq:section6-E-small} imply
$|\Sigma_\lambda^-|\to0$ as $\lambda\uparrow{\lambda_0}$. Combining this with
\eqref{eq:section6-uniform-kelvin-bound}, we obtain
\begin{equation}\label{eq:section6-small-coefficient}
    \int_{\Sigma_\lambda^-}u_\lambda^{\tau+1}\,dy
    \to 0
    \quad\text{as }\lambda\uparrow{\lambda_0}.
\end{equation}
For $\lambda$ sufficiently close to ${\lambda_0}$, the coefficient on the right-hand side of
\eqref{eq:section6-HLS} is strictly smaller than one. Hence
$\Sigma_\lambda^-=\varnothing$, which proves \eqref{eq:section6-continuation}.
However, \eqref{eq:section6-continuation} contradicts the infimum definition of \(\lambda_{0}\).
Thus, we have $\lambda_{0}=0$. Consequently,
\begin{equation}\label{eq:section6-all-spheres}
    u\ge u_\lambda
    \quad\text{in }B_\lambda\setminus\{0\} ~~\text{for every }0<\lambda\le1.
\end{equation}

Fix $\theta\in\mathbb S^{n-1}$. By \eqref{eq:section6-interior-positivity}, we have $u(\theta/2)>0$.
For $0<r<\frac12$, let $x=r\theta$ and $\lambda=\sqrt{r/2}$.
Then $x\in B_\lambda\setminus\{0\}$ and $x^\lambda=\theta/2$.
By \eqref{eq:section6-all-spheres}, we obtain
\[
\begin{aligned}
u(r\theta)
\ge
\left(\frac{\lambda}{r}\right)^{n-2\sigma}
u(x^\lambda)=
\left({\frac {1}{2r}}\right)^{\frac{n-2\sigma}{2}}
u\left( \frac\theta2 \right) \to +\infty \text{ as } r \to 0.
\end{aligned}
\]
This contradicts the continuity of $u$ at the origin. Hence, equation \eqref{sobolev eq 7} admits no nontrivial nonnegative solution. The proof of Theorem \ref{thm:no-asymptotic} is completed.
\end{proof}

\begin{proof}[Proof of Corollary \ref{thm sphere}] It suffices to show the non-existence of positive solutions of \eqref{sobolev eq 7} under the assumptions that
\begin{equation}\label{eq 5.4}
|x|^{\,n-2\sigma} u(x) \to M > 0 \quad \text{as } |x| \to \infty
\end{equation}
and
\begin{equation}\label{eq 5.5}
\begin{cases}
K(x)=K(|x|) \in C (\mathbb{R}^n), \\
K(r) > 0 \text{ is nonincreasing for } 0 \leq r < 1, \\
K(r) \le 0  \text{ for } r > 1 .
\end{cases}
\end{equation}
Under these assumptions, we only need to prove that the solution $u$ of \eqref{sobolev eq 7} also satisfies the integral equation
\begin{equation}\label{26Ineuq}
u(x)=C_{n, \sigma} \int_{\mathbb{R}^n} \frac{K(y) u^\frac{n+2\sigma}{n-2\sigma}(y)}{|x-y|^{n-2 \sigma}} d y \quad \text { for }  x \in \mathbb{R}^n.
\end{equation}
The rest of the proof follows exactly as in Theorem \ref{thm:no-asymptotic}. By \eqref{eq 5.4}, we have
\begin{equation}\label{eq 5.1}
\int_{\mathbb{R}^{n}} \frac{u(x)}{1+|x|^{q}} \, dx
\leq C + \int_{B_{1}^{c}} \frac{M}{|x|^{n-2\sigma}(1+|x|^{q})} \, dx < \infty,
\quad \forall \, q > \frac{n+2\sigma}{2}.
\end{equation}
For small \(\varepsilon > 0\), by Lemma \ref{lemma remo} we use \(\psi \in C^{\infty}(\mathbb{R}^{n})\) with
\[
\psi(x) = \frac{1}{|x|^{\frac{n-2\sigma}{2}+\varepsilon}} \quad \text{on } B_{1}^{c}
\]
as a test function to the differential equation \eqref{sobolev eq 7}. Then, similar to the proof of \eqref{eq 1} in Proposition \ref{prop u lift}, we have that for any $q > \frac{n-2\sigma}{2}$,
\begin{equation}\label{eq 5.2}
\left| \int_{\mathbb{R}^{n}} \frac{K(x) u^{\frac{n+2\sigma}{n-2\sigma}}(x)}{1+|x|^{q}} \, dx \right|
\leq \left| \int_{\mathbb{R}^{n}} u(x)(-\Delta)^{\sigma} \psi(x) \, dx \right|
\leq \int_{\mathbb{R}^{n}} \frac{u(x)}{1+|x|^{q+2\sigma}} \, dx < \infty.
\end{equation}
This and \eqref{eq 5.1} coincide with the conclusion of Proposition \ref{prop u lift} when $p = \frac{n+2\sigma}{n-2\sigma}$ and $\alpha = 0$. Using a similar proof as in Proposition \ref{prop v} and Theorem \ref{thm equivalence}, we obtain that
$u$ also satisfies the integral equation \eqref{26Ineuq}. This proves Corollary \ref{thm sphere}.
\end{proof}

\appendix

\section*{Appendix}
\addcontentsline{toc}{section}{Appendix}

\newcommand{\appendixsubsection}[1]{%
  \refstepcounter{section}%
  \subsection*{\thesection\quad #1}%
  \addcontentsline{toc}{subsection}{\protect\numberline{\thesection}#1}%
}

\appendixsubsection{\texorpdfstring{Values of $\int_{\partial B_r} g_m(u_0,u_0)$ in Lemma \ref{lemma 3.4}}{Values of gm (u0,u0) in Lemma}} \label{sec gm}

In this appendix, we present a detailed calculation for $\int_{\partial B_r} g_m(u_0,u_0)$ in Lemma \ref{lemma 3.4}. By Guo-Peng-Yan \cite[(3.12), (3.14), (3.15)]{guo2015existencelocaluniquenessbubbling}, for any smooth function $u$ in a bounded domain,
\begin{align*}
   2\int_{\Omega}(x\cdot \nabla u)(-\Delta)^m u &= -2\int_{\partial\Omega} (y\cdot \nabla u)\frac{\partial (-\Delta)^{m-1} u}{\partial \nu}
   + 2\int_{\partial\Omega}  \Big(y\cdot \nabla \frac{\partial u}{\partial \nu}\Big) (-\Delta)^{m-1} u\\
   &\quad + (n-2m+2) \int_{\partial\Omega} (-\Delta)^{m-1} u \frac{\partial u}{\partial \nu}
   - (n-2m) \int_{\partial\Omega} u \frac{\partial (-\Delta)^{m-1} u}{\partial \nu} \\
   &\quad + \int_{\partial\Omega} g_{m-2}(-\Delta u, -\Delta u)
   - (n-2m) \int_{\Omega} u (-\Delta)^m u.
\end{align*}
Hence, $\int_{\partial B_r} g_m(u,u)$ can be represented iteratively by
\begin{equation}\label{eq iterated 1}
\begin{aligned}
    \int_{\partial\Omega} g_m(u,u)
    &\overset{\eqref{eq 4.9}}{=}  2\int_{B_r}(x\cdot \nabla u)(-\Delta)^m u + (n-2m) \int_{B_r} u(-\Delta)^m u\\
    &=-2\int_{\partial\Omega} (y\cdot \nabla u)\frac{\partial (-\Delta)^{m-1} u}{\partial \nu}
     + 2\int_{\partial\Omega} \Big(y\cdot \nabla \frac{\partial u}{\partial \nu}\Big) (-\Delta)^{m-1} u \\
    &\quad + (n-2m+2) \int_{\partial\Omega} (-\Delta)^{m-1} u \frac{\partial u}{\partial \nu}
     - (n-2m) \int_{\partial\Omega} u \frac{\partial (-\Delta)^{m-1} u}{\partial \nu} \\
    &\quad + \int_{\partial\Omega} g_{m-2}(-\Delta u, -\Delta u),\qquad m=3,4,...
\end{aligned}
\end{equation}
Take $\Omega=B_r$. Assume $x\in\partial B_r$ and $\nu$ is the unit outward normal. Then for $s>0$, we have
\begin{equation}\label{eq basic 1}
\begin{aligned}
    (-\Delta)^m |x|^{-s} &= \frac{F(s,m)}{|x|^{s+2m}}, \quad \text{where }F(s,m) = \prod_{i=1}^m (s+2i-2)(n-2i-s), \\
    \frac{\partial (-\Delta)^{m-1}|x|^{-s}}{\partial\nu} &= -\frac{(s+2m-2)F(s,m-1)}{|x|^{s+2m-1}}.
\end{aligned}
\end{equation}
Combining \eqref{eq basic 1} with \eqref{eq iterated 1} yields
\begin{equation*}
\begin{aligned}
    \int_{\partial B_r} g_m(|x|^{-s},|x|^{-s})
    &= -2s(s+2m-2)F(s,m-1)|\mathbb{S}^{n-1}| r^{n-2m-2s} \\
    &\quad + 2s(s+1)F(s,m-1)|\mathbb{S}^{n-1}| r^{n-2m-2s} \\
    &\quad - (n-2m+2)s F(s,m-1)|\mathbb{S}^{n-1}| r^{n-2m-2s} \\
    &\quad + (n-2m)(s+2m-2) F(s,m-1)|\mathbb{S}^{n-1}| r^{n-2m-2s} \\
    &\quad + \int_{\partial B_r} g_{m-2}\Big(\frac{s(s+2-n)}{|x|^{s+2}},\frac{s(s+2-n)}{|x|^{s+2}}\Big).
\end{aligned}
\end{equation*}
Simplifying,
\begin{equation*}
\begin{aligned}
   \int_{\partial B_r} g_m(|x|^{-s},|x|^{-s})
   &= 2(m-1)(n-2m-2s)F(s,m-1)|\mathbb{S}^{n-1}| r^{n-2m-2s} \\
   &\quad + s^2(s+2-n)^2 \int_{\partial B_r} g_{m-2}(|x|^{-(s+2)},|x|^{-(s+2)}).
\end{aligned}
\end{equation*}
Substitute $s=\frac{n-2m}{2}$ to obtain the iteration
\begin{equation}\label{eq iterate 2}
\begin{aligned}
    &\int_{\partial B_r} g_m(|x|^{-\frac{n-2m}{2}},|x|^{-\frac{n-2m}{2}})\\
    &\qquad= \Big(\frac{n-2m}{2}\Big)^2\Big(\frac{n+2(m-2)}{2}\Big)^2
      \int_{\partial B_r} g_{m-2}(|x|^{-\frac{n-2(m-2)}{2}},|x|^{-\frac{n-2(m-2)}{2}}).
      \end{aligned}
\end{equation}
We compute the first two terms explicitly. By \cite[(3.11), (3.13)]{guo2015existencelocaluniquenessbubbling},
\begin{equation*}
\begin{aligned}
   & \int_{\partial B_r} g_1(|x|^{-s},|x|^{-s}) \\
    &\qquad= -2\int_{\partial B_r} \frac{\partial |x|^{-s}}{\partial \nu} (x\cdot \nabla |x|^{-s})
      + \int_{\partial B_r} (x\cdot\nu)|\nabla |x|^{-s}|^2
      - (n-2)\int_{\partial B_r} \frac{\partial |x|^{-s}}{\partial \nu}|x|^{-s},
      \end{aligned}
\end{equation*}
and
\begin{equation*}
\begin{aligned}
    \int_{\partial B_r} g_2(|x|^{-s},|x|^{-s})
    = &-2\int_{\partial B_r} \frac{\partial (-\Delta |x|^{-s})}{\partial \nu}(x\cdot \nabla |x|^{-s})
      + 2\int_{\partial B_r} (-\Delta |x|^{-s}) (x\cdot \nabla \frac{\partial |x|^{-s}}{\partial \nu})\\
      + &\frac{n-2}{2}\int_{\partial B_r} (-\Delta |x|^{-s})\frac{\partial |x|^{-s}}{\partial \nu}
      + \int_{\partial B_r} (x\cdot\nu)(-\Delta |x|^{-s})^2\\
      - &(n-4)\int_{\partial B_r} |x|^{-s} \frac{\partial (-\Delta |x|^{-s})}{\partial \nu}.
      \end{aligned}
\end{equation*}
Let $s=\frac{n-2m}{2}$ for $m=1,2$ respectively. Then
\begin{equation}\label{g1}
    \int_{\partial B_r} g_1(|x|^{-\frac{n-2}{2}},|x|^{-\frac{n-2}{2}})
    = \Big(\frac{n-2}{2}\Big)^2 |\mathbb{S}^{n-1}|,
\end{equation}
\begin{equation}\label{g2}
    \int_{\partial B_r} g_2(|x|^{-\frac{n-4}{2}},|x|^{-\frac{n-4}{2}})
    = \Big(\frac{n}{2}\Big)^2 \Big(\frac{n-4}{2}\Big)^2 |\mathbb{S}^{n-1}|.
\end{equation}
One may notice that \eqref{g1} and \eqref{g2} are independent of $r$. Combining \eqref{eq iterate 2}-\eqref{g2} and scaling,
\begin{equation*}
    \int_{\partial B_r} g_m(u_0,u_0)
    = |\mathbb{S}^{n-1}| M_0^2 \prod_{i=0}^{m-1}\Big(\frac{n-2m+4i}{2}\Big)^2,\qquad m=1,2,3,...
\end{equation*}

\appendixsubsection{Definition and some properties of the hypergeometric function}

For readers' convenience, we provide the definition and some properties of the hypergeometric function, which is used in Lemma \ref{lemma 3.5}.

\begin{definition}[{\cite[Eq.~15.1.1]{abramowitz_handbook_2013}}]
The hypergeometric function ${}_2F_1(a,b;c;z)$ is defined by the {Gauss series}:
\begin{equation*}
{}_2F_1(a,b;c;z)
= 1 + \frac{ab}{c} z
  + \frac{a(a+1)b(b+1)}{c(c+1)2!} z^2 + \cdots
= \frac{\Gamma(c)}{\Gamma(a)\Gamma(b)}
  \sum_{s=0}^{\infty} \frac{\Gamma(a+s)\Gamma(b+s)}{\Gamma(c+s)\, s!} \, z^s .
\end{equation*}
The series converges for $|z|<1$, and defines ${}_2F_1(a,b;c;z)$ elsewhere by analytic continuation.
In general, ${}_2F_1(a,b;c;z)$ is undefined when $c=0,-1,-2,\dots$.
\end{definition}

\begin{proposition}
\begin{enumerate}
    \item From the definition,
    \begin{equation}\label{eq:hyper_zero}
        {}_2F_1(a,b;c;0) = 1.
    \end{equation}

    \item If $\Re(c-a-b) > 0$, then
    \begin{equation}\label{eq:hyper_one}
        {}_2F_1(a,b;c;1)
        = \frac{\Gamma(c)\,\Gamma(c-a-b)}{\Gamma(c-a)\,\Gamma(c-b)}.
    \end{equation}
    (see \cite[Eq.~15.1.20]{abramowitz_handbook_2013})

    \item The derivative formula:
    \begin{equation}\label{eq:hyper_derivative}
        \frac{d}{dz} {}_2F_1(a,b;c;z)
        = \frac{ab}{c} \, {}_2F_1(a+1, b+1; c+1; z).
    \end{equation}
    (see \cite[Eq.~15.2.1]{abramowitz_handbook_2013})

    \item Euler's integral representation:
    \begin{equation}\label{eq:hyper_integral}
        {}_2F_1(a,b;c;z)
        = \frac{\Gamma(c)}{\Gamma(b)\,\Gamma(c-b)}
          \int_0^1 t^{\,b-1}(1-t)^{\,c-b-1}(1-zt)^{-a}\,dt .
    \end{equation}
    valid for $|\arg(1-z)|<\pi$ and $\Re(c)>\Re(b)>0$.
    (see \cite[Eq.~15.3.1]{abramowitz_handbook_2013})

    \item Transformation formulas:
    \begin{equation}\label{eq:hyper_transformation}
        {}_2F_1(a,b;c;z)
        = (1-z)^{-a} {}_2F_1\!\left(a,\, c-b;\, c;\, \frac{z}{z-1}\right)
        = (1-z)^{-b} {}_2F_1\!\left(c-a,\, b;\, c;\, \frac{z}{z-1}\right),
    \end{equation}
    valid for $|\arg(1-z)| < \pi$. Moreover,
    \begin{equation}\label{eq:hyper_Gauss trans}
        {}_2F_1(a,b;c;z)
        = (1-z)^{c-a-b} {}_2F_1(c-a,\, c-b;\, c;\, z).
    \end{equation}
    (see \cite[Eq.~15.3.3-15.3.5]{abramowitz_handbook_2013})

    \item The hypergeometric differential equation:
    \begin{equation}\label{eq:hyper_diffeq}
        z(1-z) \frac{d^2}{dz^2} {}_2F_1(a,b;c;z)
        + \big(c - (a+b+1)z \big) \frac{d}{dz} {}_2F_1(a,b;c;z)
        - ab\, {}_2F_1(a,b;c;z) = 0,
    \end{equation}
    valid for $c \notin \{0,-1,-2,\dots\}$.
    (see \cite[Section 15.5]{abramowitz_handbook_2013})
\end{enumerate}
\end{proposition}

\appendixsubsection{Computation of the Pohozaev integral in the fractional case}\label{CC}
This appendix is devoted to computing the Pohozaev integral in the fractional case, which is used in Lemma \ref{lemma 3.5}. We first give a standard integral identity based on the Gamma functions.

\begin{proposition}\label{appendix 1}
    For $A,B>0$ and $\Re(\mu)>0$, $\Re(\nu)>0$, we have
\begin{equation*}
\frac{\Gamma(\mu+\nu)}{\Gamma(\mu)\Gamma(\nu)}
  \int_0^1 s^{\mu-1} (1-s)^{\nu-1} \,
  \big[sA + (1-s)B\big]^{-(\mu+\nu)} \, ds=\frac{1}{A^{\mu} B^{\nu}}
.
\end{equation*}
\end{proposition}

\begin{proof}
We start from the Gamma integral
\begin{equation*}
\int_0^\infty t^{z-1} e^{-At}\, dt = \int_0^\infty \left(\frac{t}{A}\right)^{z-1} e^{-t}\, \frac{dt}{A} = \frac{1}{A^z} \int_0^\infty t^{z-1} e^{-t}\, dt = \frac{\Gamma(z)}{A^z}.
\end{equation*}
Hence
\begin{equation*}
\frac{1}{A^\mu} = \frac{1}{\Gamma(\mu)} \int_0^\infty t^{\mu-1} e^{-At}\, dt.
\end{equation*}
Similarly,
\begin{equation*}
\frac{1}{A^\mu B^\nu} = \frac{1}{\Gamma(\mu)\Gamma(\nu)} \int_0^\infty \int_0^\infty t^{\mu-1} s^{\nu-1} e^{-At-Bs}\, dt\, ds.
\end{equation*}
Now make the change of variables $t = u v$, $s = u(1-v)$ with $u \in (0,\infty)$, $v\in(0,1)$, whose Jacobian satisfies $\partial(t,s)/\partial(u,v) = u$. Thus
\begin{align*}
\frac{1}{A^\mu B^\nu} &= \frac{1}{\Gamma(\mu)\Gamma(\nu)} \int_0^\infty \int_0^1 (uv)^{\mu-1} (u(1-v))^{\nu-1} e^{-Au v - B u(1-v)} u \, dv\, du \\
&= \frac{1}{\Gamma(\mu)\Gamma(\nu)} \int_0^1 v^{\mu-1}(1-v)^{\nu-1} \left( \int_0^\infty u^{\mu+\nu-1} e^{-(Av+B(1-v))u}\, du \right) dv.
\end{align*}
Evaluating the inner integral gives
\begin{equation*}
\int_0^\infty u^{\mu+\nu-1} e^{-(Av+B(1-v))u}\, du = \frac{\Gamma(\mu+\nu)}{(Av+B(1-v))^{\mu+\nu}}.
\end{equation*}
Therefore
\begin{equation*}
\frac{1}{A^\mu B^\nu} = \frac{\Gamma(\mu+\nu)}{\Gamma(\mu)\Gamma(\nu)} \int_0^1 v^{\mu-1}(1-v)^{\nu-1} (Av+(1-v)B)^{-(\mu+\nu)} dv.
\end{equation*}
\end{proof}

\begin{proposition}\label{appendix 2}
For $m>n/2$ and $f(x)>0$, we have
\begin{equation*}
\int_{\mathbb{R}^n} \frac{dy}{(|y|^2 + f(x))^m}
= \frac{1}{2} |\mathbb{S}^{n-1}|\,
  \frac{\Gamma\!\left(m-\frac{n}{2}\right)\Gamma\!\left(\frac{n}{2}\right)}
       {\Gamma(m)}\,
  (f(x))^{\frac{n}{2}-m}.
\end{equation*}
\end{proposition}

\begin{proof}
We compute
\begin{align*}
\text{LHS} &= |\mathbb{S}^{n-1}| \int_0^\infty \frac{r^{n-1}}{(r^2+f(x))^m}\, dr = |\mathbb{S}^{n-1}| \int_0^\infty \frac{t^{\frac{n}{2}-1}}{(t+f(x))^m}\, \frac{dt}{2t^{1/2}} \qquad (r^2=t) \\
&= \frac{1}{2} |\mathbb{S}^{n-1}| \int_0^\infty \frac{t^{\frac{n}{2}-1}}{(t+f(x))^m}\, dt = \frac{1}{2} |\mathbb{S}^{n-1}| (f(x))^{\frac{n}{2}-m} \int_0^\infty \frac{s^{\frac{n}{2}-1}}{(1+s)^m}\, ds \qquad (t=f(x) s).
\end{align*}
The last integral is a Beta function:
\begin{equation*}
\int_0^\infty \frac{s^{\frac{n}{2}-1}}{(1+s)^m}\, ds = B\!\left(\frac{n}{2}, m-\frac{n}{2}\right)= \frac{\Gamma\!\left(m-\frac{n}{2}\right)\Gamma\!\left(\frac{n}{2}\right)}  {\Gamma(m)}.
\end{equation*}
Thus
\begin{equation*}
\text{LHS} = \frac{1}{2} |\mathbb{S}^{n-1}| (f(x))^{\frac{n}{2}-m }\frac{\Gamma\!\left(m-\frac{n}{2}\right)\Gamma\!\left(\frac{n}{2}\right)}{\Gamma(m)}.
\end{equation*}
\end{proof}

The following lemma provides an explicit formula for the extension $U_0(x, t)$ via a hypergeometric function.

\begin{lemma}\label{appendix U}
    Let $u_0(x)=M_0|x|^{-\frac{n-2\sigma}{2}}$ and $U_0(x,t) = (\mathcal{P}_\sigma * u_0)(x,t) $, where $\mathcal{P}_\sigma$ is defined as in Lemma \ref{lemma 3.5}. Then
    \begin{equation*}
U_0(x,t) = \frac{\Gamma^2(\frac{n+2\sigma}{4})}{\Gamma(\frac{n}{2})\Gamma(\sigma)}
  M_0 (|x|^2+t^2)^{-\frac{n-2\sigma}{4}}
  {}_2F_1\!\left(\frac{n-2\sigma}{4},\frac{n-2\sigma}{4};\frac{n}{2};\frac{|x|^2}{|x|^2+t^2}\right).
\end{equation*}
\end{lemma}

\begin{proof}

We have
\begin{equation*}
U_0(x,t) = P_\sigma * u_0(x,t)
= \beta(n, \sigma) \int_{\mathbb{R}^n}
\frac{t^{2\sigma}}{(|y|^2 + t^2)^{\frac{n+2\sigma}{2}}}
\frac{M_0}{|x-y|^{\frac{n-2\sigma}{2}}} \, dy.
\end{equation*}
By Proposition \ref{appendix 1},
\begin{equation}\label{eq 2.6}
\frac{1}{A^\mu B^\nu}
= \frac{\Gamma(\mu+\nu)}{\Gamma(\mu)\Gamma(\nu)}
\int_0^1 s^{\mu-1} (1-s)^{\nu-1} (sA + (1-s)B)^{-(\mu+\nu)} \, ds.
\end{equation}
Let $A = |x-y|^2 + t^2$, $B = |y|^2$, $\mu = \frac{n+2\sigma}{2}$, $\nu = \frac{n-2\sigma}{4}$. Then
\begin{equation*}
sA + (1-s)B = |y-sx|^2 + s(1-s)|x|^2 + s t^2.
\end{equation*}
Plugging \eqref{eq 2.6} and integrating over $y$, we obtain
\begin{equation*}
\setlength{\jot}{10pt}
\begin{aligned}
&\int_{\mathbb{R}^n}
   \frac{t^{2\sigma}}{(|x-y|^2 + t^2)^{\frac{n+2\sigma}{2}}}
   \frac{1}{|y|^{\frac{n-2\sigma}{4}}} \, dy\\
&= \frac{\Gamma(\frac{3}{4}n + \frac{\sigma}{2}) t^{2\sigma}}
        {\Gamma(\frac{n+2\sigma}{2}) \Gamma(\frac{n-2\sigma}{4})}
   \!\int_0^1 s^{\frac{n+2\sigma}{2}-1} (1-s)^{\frac{n-2\sigma}{4}-1}\\
  &\qquad\times \left[\int_{\mathbb{R}^n} (|y-sx|^2 + s(1-s)|x|^2 + st^2)^{-\frac{3}{4}n-\frac{\sigma}{2}} dy \right] ds
\\
&\stackrel{(\ref{appendix 2})}{=}
\frac{1}{2} |\mathbb{S}^{n-1}|\frac{\Gamma(\frac{n}{2}) \Gamma(\frac{n+2\sigma}{4})}
     {\Gamma(\frac{n+2\sigma}{2}) \Gamma(\frac{n-2\sigma}{4})}
\, t^{2\sigma }
\int_0^1 s^{\frac{n+2\sigma}{4}-1} (1-s)^{\frac{n-2\sigma}{4}-1}
((1-s)|x|^2 + t^2)^{-\frac{n+2\sigma}{4}} ds
\\
&\stackrel{u = 1-s}{=}
\frac{1}{2} |\mathbb{S}^{n-1}|
\frac{\Gamma(\frac{n}{2}) \Gamma(\frac{n+2\sigma}{4})}
     {\Gamma(\frac{n+2\sigma}{2}) \Gamma(\frac{n-2\sigma}{4})}
t^{\sigma - \frac{n}{2}}
\int_0^1 u^{\frac{n-2\sigma}{4}-1} (1-u)^{\frac{n}{4}+\frac{\sigma}{2}-1}
\left( 1 + u \frac{|x|^2}{t^2} \right)^{-\frac{n+2\sigma}{4}} du
\\
&\stackrel{\eqref{eq:hyper_integral}}{=}
\frac{1}{2} |\mathbb{S}^{n-1}| t^{\sigma - \frac{n}{2}}
\frac{ \Gamma^2(\frac{n+2\sigma}{4})}
     {\Gamma(\frac{n+2\sigma}{2}) }
{}_2F_1\!\!\left(\frac{n+2\sigma}{4}, \frac{n-2\sigma}{4}; \frac{n}{2}; -\frac{|x|^2}{t^2}\right).
\end{aligned}
\end{equation*}
Since
\begin{equation*}
\beta(n, \sigma) \int_{\mathbb{R}^n} (|y|^2+1)^{-\frac{n+2\sigma}{2}} dy = 1,
\end{equation*}
Proposition \ref{appendix 2} gives
\begin{equation*}
\beta(n, \sigma)= \frac{2}{|\mathbb{S}^{n-1}|} \frac{\Gamma(\frac{n}{2}+\sigma)}{\Gamma(\frac{n}{2}) \Gamma(\sigma)}.
\end{equation*}
Therefore,
\begin{equation*}
U_0(x,t) = \frac{\Gamma^2(\frac{n+2\sigma}{4})}{\Gamma(\frac{n}{2}) \Gamma(\sigma)}
  M_0 \, t^{-\frac{n-2\sigma}{2}}
  {}_2F_1\!\left(\frac{n+2\sigma}{4}, \frac{n-2\sigma}{4}; \frac{n}{2}; -\frac{|x|^2}{t^2}\right).
\end{equation*}
By \eqref{eq:hyper_transformation},
\begin{equation*}
U_0(x,t) = \frac{\Gamma^2(\frac{n+2\sigma}{4})}{\Gamma(\frac{n}{2})\Gamma(\sigma)}
  M_0 t^{2\sigma} (|x|^2+t^2)^{-\frac{n+2\sigma}{4}}
  {}_2F_1\!\left(\frac{n+2\sigma}{4},\frac{n+2\sigma}{4};\frac{n}{2};\frac{|x|^2}{|x|^2+t^2}\right).
\end{equation*}
Using \eqref{eq:hyper_Gauss trans}, we finally get
\begin{equation*}
U_0(x,t) = \frac{\Gamma^2(\frac{n+2\sigma}{4})}{\Gamma(\frac{n}{2})\Gamma(\sigma)}
  M_0 (|x|^2+t^2)^{-\frac{n-2\sigma}{4}}
  {}_2F_1\!\left(\frac{n-2\sigma}{4},\frac{n-2\sigma}{4};\frac{n}{2};\frac{|x|^2}{|x|^2+t^2}\right).
\end{equation*}
\end{proof}

The following lemma yields a reduced expression for the Pohozaev identity for $u_0(x)=M_0|x|^{-\frac{n-2\sigma}{2}}$, as several boundary terms are shown to vanish.

\begin{lemma}\label{reduction}
Denote
\begin{equation*}
        \begin{aligned}
          P(\sigma,R,u) =  & N_{n,\sigma} \frac{n-2\sigma}{2n} R \int_{\partial B_R} K u^{\frac{2n}{n-2\sigma}}  +  \left[ \frac{n-2\sigma}{2} \int_{\partial'' \mathcal{B}_R} t^b U \partial_\nu  U  \right.\\
&+ \left.\int_{\partial'' \mathcal{B}_R} t^b \partial_\nu  U \langle X, \nabla U \rangle
+ \int_{\partial'' \mathcal{B}_R} Q_0(U)\right] \qquad \text{for any }R>0.
        \end{aligned}
    \end{equation*}
      Let $u_0(x)=M_0|x|^{-\frac{n-2\sigma}{2}}$ and $U_0(x,t) = (\mathcal{P}_\sigma * u_0)(x,t) $, where $\mathcal{P}_\sigma$ is defined as in Lemma \ref{lemma 3.5}. Then
      \begin{equation*}
\begin{aligned}
P(\sigma,R,u_0)
&=  N_{n,\sigma} \frac{n-2\sigma}{2n} R \int_{\partial B_R} K u_0^{\frac{2n}{n-2\sigma}}
+ \int_{\partial'' \mathcal{B}_R} Q_0(U_0).
\end{aligned}
\end{equation*}
\end{lemma}

\begin{proof}
By Lemma~\ref{appendix U},
\begin{equation*}
U_0(x,t) = \frac{\Gamma^2\!\left(\frac{n+2\sigma}{4}\right)}{\Gamma\!\left(\frac{n}{2}\right)\Gamma(\sigma)} M_0 (|x|^2+t^2)^{-\frac{n-2\sigma}{4}} {}_2F_1\!\left(\frac{n-2\sigma}{4},\frac{n-2\sigma}{4};\frac{n}{2};\frac{|x|^2}{|x|^2+t^2}\right).
\end{equation*}
Let
\begin{equation*}
r^2 = |x|^2 + t^2, \quad z = \frac{|x|^2}{r^2}, \quad F(z) = {}_2F_1\!\left(\frac{n-2\sigma}{4}, \frac{n-2\sigma}{4}; \frac{n}{2}; z\right)
\end{equation*}
and
\begin{equation*}
    C_0 = \frac{\Gamma^2\!\left(\frac{n+2\sigma}{4}\right)}{\Gamma\!\left(\frac{n}{2}\right)\Gamma(\sigma)}.
\end{equation*}
Then
\begin{equation*}
U_0(x,t) = C_0 M_0 r^{-\frac{n-2\sigma}{2}} F(z).
\end{equation*}
For the derivatives,
\begin{equation*}
\partial_{x_i} U_0(x,t) = C_0 M_0 r^{-\frac{n-2\sigma}{2}-2} x_i \left(-\frac{n-2\sigma}{2} F(z) + 2 (1-z) F'(z)\right),
\end{equation*}
\begin{equation}\label{eq partial_t U}
\partial_t U_0(x,t) = C_0 M_0 r^{-\frac{n-2\sigma}{2}-2} t \left(-\frac{n-2\sigma}{2} F(z) - 2 z F'(z)\right).
\end{equation}
Thus
\begin{equation*}
\begin{aligned}
|\nabla_{x,t} U_0|^2 &= |\nabla_x U_0|^2 + |\partial_t U_0|^2 \\
&= C_0^2 M_0^2 r^{-(n-2\sigma)-2}\left[\left(\frac{n-2\sigma}{2}\right)^2 F^2(z) + 4 z(1-z) (F'(z))^2\right].
\end{aligned}
\end{equation*}
For $(x,t) \in \partial'' \mathcal{B}_R$,
\begin{equation*}
\partial_\nu U_0 = \frac{d}{dr}\big(C_0 M_0 r^{-\frac{n-2\sigma}{2}} F(z)\big)\Big|_{r=R} = -\frac{n-2\sigma}{2} \frac{1}{R} U_0.
\end{equation*}
Hence
\begin{equation*}
\langle X, \nabla U_0 \rangle = R\, \partial_\nu U_0 = -\frac{n-2\sigma}{2} U_0, \qquad X = (x,t) \in \partial'' \mathcal{B}_R.
\end{equation*}
It follows that
\begin{equation*}
\frac{n-2\sigma}{2} \int_{\partial'' \mathcal{B}_R} t^b  U_0\partial_\nu  U_0 + \int_{\partial'' \mathcal{B}_R} t^b \partial_\nu  U_0 \langle X, \nabla U_0 \rangle = 0.
\end{equation*}
Therefore, $P(\sigma,R,u_0)$ reduces to
\begin{equation*}
\begin{aligned}
P(\sigma,R,u_0)
&=  N_{n,\sigma} \frac{n-2\sigma}{2n} R \int_{\partial B_R} K u_0^{\frac{2n}{n-2\sigma}}
+ \int_{\partial'' \mathcal{B}_R} Q_0(U_0).
\end{aligned}
\end{equation*}
\end{proof}

We compute the constants $N_{n,\sigma}$, $M_0$ and the integral $\int_{\partial'' \mathcal{B}_R }Q_0(U_0)$ involved in Lemma~\ref{lemma 3.5}. By \eqref{q 1} and Lemma \ref{appendix U}, we have
\begin{equation*}
\begin{aligned}
\int_{\partial'' \mathcal{B}_R} Q_0(U_0)
&= -\frac{1}{2} \int_{\partial'' \mathcal{B}_R} t^{1-2\sigma} |\nabla U_0|^2 X^k \nu_k \, d\sigma \\
&= \frac{-C_0^2 M_0^2}{2}  \int_{\{|x|^2+t^2=R^2,\, t>0\}}
   R t^{1-2\sigma}R^{-(n-2\sigma)-2} \\
   &\qquad\times  \Bigg[
   \Big(\frac{n-2\sigma}{2}\Big)^2  F^2(z)
   + 4  z(1-z)(F'(z))^2
   \Bigg] \, dS \\
&= \frac{-C_0^2 M_0^2}{2 R^{n}}\int_{B_R}
   \Big(1 - \frac{|x|^2}{R^2}\Big)^{-\sigma}\\
   &\qquad\times  \Bigg[ \Big(\frac{n-2\sigma}{2}\Big)^2  F^2\!\Big(\frac{|x|^2}{R^2}\Big)
   + 4 \frac{|x|^2 (R^2-|x|^2)}{R^4} \big(F'(\frac{|x|^2}{R^2})\big)^2 \Bigg] dx \\
&= \frac{-C_0^2 M_0^2}{2 R^{n}} |\mathbb{S}^{n-1}| \int_0^R
   r^{n-1} \Big(1-\frac{r^2}{R^2}\Big)^{-\sigma} \\
   &\qquad\times
   \Bigg[ \Big(\frac{n-2\sigma}{2}\Big)^2 F^2\!\Big(\frac{r^2}{R^2}\Big)
   + 4 \frac{r^2(R^2-r^2)}{R^4} \big(F'(\frac{r^2}{R^2})\big)^2 \Bigg] dr.
\end{aligned}
\end{equation*}
With the substitution $s = \frac{r^2}{R^2}$, this becomes
\begin{equation}\label{eq A.1}
\int_{\partial'' \mathcal{B}_R} Q_0(U_0)
= -\frac{C_0^2 M_0^2}{4} |\mathbb{S}^{n-1}| \int_0^1
   s^{\frac{n-2}{2}} (1-s)^{-\sigma}
   \Bigg[ \Big(\frac{n-2\sigma}{2}\Big)^2 F^2(s)
   + 4s(1-s)\big(F'(s)\big)^2 \Bigg] ds.
\end{equation}
By \eqref{eq:hyper_diffeq},
\begin{equation*}
s(1-s)F''(s) + \Big( \frac{n}{2} - \Big(\frac{n-2\sigma}{2}+1\Big)s \Big) F'(s)
- \Big( \frac{n-2\sigma}{4} \Big)^2 F(s) = 0,
\end{equation*}
which is equivalent to
\begin{equation*}
\Big( s^{\frac{n}{2}} (1-s)^{1-\sigma} F'(s) \Big)'
= \Big( \frac{n-2\sigma}{4} \Big)^2 s^{\frac{n}{2}-1}(1-s)^{-\sigma} F(s).
\end{equation*}
Multiplying by $F$, using Leibniz's rule, and integrating over $(0,1)$, we obtain
\begin{equation}\label{eq A.2}
\Big[ s^{\frac{n}{2}} (1-s)^{1-\sigma} F(s)F'(s) \Big]_{s=0}^1
= \int_0^1 s^{\frac{n}{2}-1}(1-s)^{-\sigma}
\Big( s(1-s)(F'(s))^2 + \Big(\frac{n-2\sigma}{4}\Big)^2 F(s)^2 \Big) \, ds.
\end{equation}
By \eqref{eq:hyper_zero}, $F(0)=1$ and $F'(0) = \big(\frac{n-2\sigma}{4}\big)^2 \frac{2}{n}$,
\begin{equation}\label{eq A.3}
s^{\frac{n}{2}} (1-s)^{1-\sigma} F(s)F'(s) \to 0, \quad s\to 0.
\end{equation}
Moreover, by \eqref{eq:hyper_one}, since $\frac{n-2\sigma}{4} + \frac{n-2\sigma}{4} < \frac{n}{2}$,
\begin{equation*}
F(1) = {}_2F_1\!\left( \frac{n-2\sigma}{4}, \frac{n-2\sigma}{4}; \frac{n}{2}; 1 \right)
= \frac{\Gamma(\frac{n}{2})\Gamma(\sigma)}{\Gamma^2(\frac{n+2\sigma}{4})}.
\end{equation*}
Also, by \eqref{eq:hyper_derivative} and \eqref{eq:hyper_Gauss trans},
\begin{equation*}
\begin{aligned}
F'(s) &= \frac{2}{n} \Big(\frac{n-2\sigma}{4}\Big)^2
 {}_2F_1\!\left( \frac{n-2\sigma}{4}+1, \frac{n-2\sigma}{4}+1; \frac{n}{2}+1; s \right)\\
 &=\frac{2}{n} \Big(\frac{n-2\sigma}{4}\Big)^2 (1-s)^{\sigma-1}
 {}_2F_1\!\left( \frac{n+2\sigma}{4}, \frac{n+2\sigma}{4}; \frac{n}{2}+1; s \right).
 \end{aligned}
\end{equation*}
Thus
\begin{equation}\label{eq A.4}
s^{\frac{n}{2}} (1-s)^{1-\sigma} F(s)F'(s)
\to \frac{\Gamma^2(\frac{n}{2}) \Gamma(\sigma)\Gamma(1-\sigma)}{\Gamma^2(\frac{n-2\sigma}{4})\Gamma^2(\frac{n+2\sigma}{4})}, \quad s \to 1.
\end{equation}
Combining \eqref{eq A.1}-\eqref{eq A.4}, we finally get
\begin{equation}\label{eq Q0}
\int_{\partial'' \mathcal{B}_R} Q_0(U_0)
= - C_0^2 M_0^2 |\mathbb{S}^{n-1}|
   \frac{\Gamma^2(\frac{n}{2}) \Gamma(\sigma)\Gamma(1-\sigma)}
   {\Gamma^2(\frac{n+2\sigma}{4}) \Gamma^2(\frac{n-2\sigma}{4})}.
\end{equation}

We now compute the constant $N_{n,\sigma}$. By \eqref{eq partial_t U} we have
\begin{equation}\label{eq e.1}
- \lim_{t \to 0} t^{1-2\sigma} \partial_t U_0(x,t)
= \lim_{t \to 0} t^{2-2\sigma} C_0 M_0 r^{-\frac{n-2\sigma}{2}-2}  \Big(
\frac{n-2\sigma}{2} F(z) + 2 zF'(z) \Big).
\end{equation}
By \eqref{eq:hyper_one},
\begin{equation*}
\lim_{t \to 0} F\!\left(\frac{|x|^2}{r^2}\right)
= \frac{\Gamma(\frac{n}{2}) \Gamma(\sigma)}{\Gamma^2(\frac{n+2\sigma}{4})}.
\end{equation*}
Moreover,
\begin{equation*}
\setlength{\jot}{10pt}
\begin{aligned}
F'(z) &= \big(\frac{n-2\sigma}{4}\big)^2 \frac{2}{n}
\, {}_2F_1\!\left(\frac{n-2\sigma}{4}+1, \frac{n-2\sigma}{4}+1; \frac{n}{2}+1; z \right) \\
&= \big(\frac{n-2\sigma}{4}\big)^2 \frac{2}{n} (1-z)^{\sigma-1}
\, {}_2F_1\!\left(\frac{n+2\sigma}{4}, \frac{n+2\sigma}{4}; \frac{n}{2}+1; z \right) \\
&= \big(\frac{n-2\sigma}{4}\big)^2 \frac{2}{n} \big(\frac{t}{r}\big)^{2\sigma-2}
\, {}_2F_1\!\left(\frac{n+2\sigma}{4}, \frac{n+2\sigma}{4}; \frac{n}{2}+1; z \right).
\end{aligned}
\end{equation*}
Using \eqref{eq:hyper_one} again,
\begin{equation}\label{eq e.2}
2C_0 M_0 \lim_{t \to 0} t^{2-2\sigma} r^{-\frac{n-2\sigma}{2}-2} zF'(z)
= 2C_0 M_0 r^{-\frac{n+2\sigma}{2}}
\frac{\Gamma(\frac{n}{2})\Gamma(1-\sigma)}{\Gamma^2(\frac{n-2\sigma}{4})}.
\end{equation}
Combining \eqref{extension 1}, \eqref{eq e.1} and \eqref{eq e.2}, we obtain
\begin{equation*}
N_{n,\sigma} (-\Delta)^\sigma M_0 |x|^{-\frac{n-2\sigma}{2}}
= 2C_0 M_0 \frac{\Gamma(\frac{n}{2})\Gamma(1-\sigma)}{\Gamma^2(\frac{n-2\sigma}{4})}
|x|^{-\frac{n+2\sigma}{2}}.
\end{equation*}
Let $C_2 = C_2(n,\sigma) > 0$ denote the constant satisfying the following equation
\begin{equation*}
(-\Delta)^\sigma |x|^{-\frac{n-2\sigma}{2}}
= C_2 |x|^{-\frac{n+2\sigma}{2}}.
\end{equation*}
Thus
\begin{equation*}
N_{n,\sigma} = 2 \frac{C_0}{C_2}
\frac{\Gamma(\frac{n}{2}) \Gamma(1-\sigma)}{\Gamma^2(\frac{n-2\sigma}{4})}.
\end{equation*}
Since $u_0(x) = M_0 |x|^{-\frac{n-2\sigma}{2}}$ solves
\begin{equation*}
(-\Delta)^\sigma u_0(x) = K(\infty)\, u_0(x)^{\frac{n+2\sigma}{n-2\sigma}},
\end{equation*}
we have
\begin{equation*}
M_0^{\frac{4\sigma}{n-2\sigma}} = \frac{C_2}{K(\infty)}.
\end{equation*}
Hence, for $R$ sufficiently large,
\begin{equation}\label{eq N}
\begin{aligned}
 N_{n,\sigma}\frac{n-2\sigma}{2n} R \int_{\partial B_R} K(\infty) u^{\frac{2n}{n-2\sigma}}
&= N_{n,\sigma}\frac{n-2\sigma}{2n} R \int_{\partial B_R} K(\infty) M_0^{\frac{2n}{n-2\sigma}} |x|^{-n} \\
&= N_{n,\sigma}\frac{n-2\sigma}{2n} |\mathbb{S}^{n-1}| M_0^{\frac{2n}{n-2\sigma}} K(\infty) \\
&= \frac{n-2\sigma}{n} |\mathbb{S}^{n-1}| M_0^{2} C_0
\frac{\Gamma(\frac{n}{2}) \Gamma(1-\sigma)}{\Gamma^2(\frac{n-2\sigma}{4})}.
\end{aligned}
\end{equation}


\bibliographystyle{amsplain}
\bibliography{refs}

\bigskip

\noindent M. Xu

\noindent  School of Mathematical Sciences, Fudan University\\
Shanghai 200433, China \\[1mm] Email:  \textsf{meiqing\_xu@fudan.edu.cn}.

\medskip

\noindent   H. Yang

\noindent  School of Mathematical Sciences, Shanghai Jiao Tong University\\
Shanghai 200240, China \\[1mm] Email:  \textsf{hui-yang@sjtu.edu.cn}

\end{document}